\documentclass{amsart}
\usepackage{latexsym, graphicx, svg, amssymb, cite, hyperref, tikz-cd}
\usepackage[normalem]{ulem}
\usepackage{multicol}
\usepackage{graphbox}

\newtheorem{theorem}{Theorem}[section] 
\newtheorem{lemma}[theorem]{Lemma}
\newtheorem{proposition}[theorem]{Proposition}

\theoremstyle{definition}

\theoremstyle{remark}
\newtheorem{remark}[theorem]{Remark}

\numberwithin{equation}{section}

\newcommand{\R}{{\mathbb R}}

\newcommand{\N}{{\mathbb N}}

\title[Cospectral vertices on finite graphs]{Cospectral vertices, walk-regular planar graphs and the echolocation problem}

\author{Shi-Lei Kong}
\address{School of Mathematics, Sichuan University, Chengdu 610065, PR China}
\email{slkong@scu.edu.cn}

\author{Emmett L. Wyman}
\address{Department of Mathematics and Statistics, Binghamton University, Binghamton NY}
\email{ewyman@binghamton.edu}

\author{Yakun Xi}
\address{School of Mathematical Sciences, Zhejiang University, Hangzhou 310027, PR China}
\email{yakunxi@zju.edu.cn}

\begin{document}
\begin{abstract}
  We study cospectral vertices on finite graphs in relation to the echolocation problem on Riemannian manifolds. First, We prove a computationally simple criterion to determine whether two vertices are cospectral. Then, we use this criterion in conjunction with a computer search to find minimal examples of various types of graphs on which cospectral but non-similar vertices exist, including minimal walk-regular non-vertex-transitive graphs, which turn out to be non-planar. Moreover, as our main result, we classify all finite 3-connected walk-regular planar graphs, proving that such graphs must be vertex-transitive.
\end{abstract}

\maketitle

\section{Introduction}

Let $G$ be a finite graph with adjacency matrix $A$. Let $\phi_1, \phi_2, \ldots, \phi_n$ be an orthonormal basis of eigenfunctions (eigenvectors) of $A$ with respective eigenvalues $\lambda_1, \lambda_2, \ldots, \lambda_n$. We say two vertices $a$ and $b$ are \emph{cospectral} if
\begin{equation}\label{eq: cospectral}
    \sum_{\lambda_j = \lambda} |\phi_j(a)|^2 = \sum_{\lambda_j = \lambda} |\phi_j(b)|^2 \qquad \text{ for all eigenvalues $\lambda$.}
\end{equation}
 All graphs considered in this article will be finite, simple, undirected, and { connected} unless otherwise specified.

{  The study of cospectral vertices has gained more attention recently, (see e.g., \cite{Statetransfer,strongly,CHAN201986,EISENBERG20192821,kempton2020characterizing}), especially due to their significant role in quantum information theory. It has been shown \cite{Statetransfer} that for a perfect quantum state transfer to occur between two vertices in an interacting qubits network, it is necessary for these vertices to be cospectral. 

We say two vertices $a$, $b$ are \emph{similar} if there is an automorphism of $G$ mapping $a$ to $b$.} Two similar vertices are clearly also cospectral. {  Recently, Kempton, Sinkovic, Smith, and Webb \cite{kempton2020characterizing} show that two vertices are cospectral if
and only if they are similar in a certain weaker sense.} We are interested in the exceptional situation where two vertices are cospectral but not similar. We come to this problem from the smooth setting, where a graph is replaced with a Riemannian manifold. {  One can formulate a notion of cospectral points that reads exactly as \eqref{eq: cospectral}, provided the adjacency matrix is replaced with the Laplace-Beltrami operator (see Remark \ref{rem: adj vs lap}).} The question of whether or not there exists a pair of cospectral points that are not in the same orbit has a musical interpretation: \emph{Can one hear, up to symmetry, where a drum is struck?} We refer to this question as the echolocation problem. The last two authors show that for most compact Riemannian manifolds, the answer is `yes' \cite{echolocation}. So far, there are no known examples of connected manifolds where the answer is `no'.

In search of such negative examples, we turn to the graph-theoretic setting, where we might use a computer search. { In fact, it is well-known that graphs with non-similar cospectral vertices exist. The first such example recorded in literature is a tree with 9 vertices observed by Schwenk \cite{schwenk1973almost}.  To find more such examples and study their properties,} we must find a way to check when two vertices are cospectral that avoids floating point errors. The criterion \eqref{eq: cospectral} is not convenient in this way, so we need an alternative.

Our method builds off of work by Godsil and Smith \cite{strongly}, who show that two vertices $a$ and $b$ of a graph are cospectral with respect to the adjacent matrix $A$,
if and only if $(A^k)_{a,a}=(A^k)_{b,b}$ for all $ k\in\mathbb N.$ In other words, $a$ and $b$ are cospectral if, for all $k \in \N$, the number of walks of length $k$ starting and ending at $a$ is equal to the number of walks of length $k$ starting and ending at $b$. See also \cite{Coutinho2023normalized, twin} for results in other settings. When it comes to computation, this condition can be used to distinguish non-cospectral vertices but would require seemingly infinitely many operations to confirm that two vertices are cospectral. In this paper, we show that, in fact, one only needs to compute a finite number of powers of $A$ to check cospectrality for a finite graph.
\begin{proposition}\label{main 1}
    Let $G$ be a graph (possibly with self-loops) with $n$ vertices and $A$ be the corresponding adjacent matrix. Then two vertices $a$ and $b$ are cospectral if and only if
    \begin{equation}
    (A^k)_{a,a}=(A^k)_{b,b},\quad\text{for all } 0\le k\le n-1.
    \end{equation}
    Moreover, this condition is sharp in the sense that the upper limit $n-1$ cannot be made smaller.
\end{proposition}
We remark that, even though Godsil and Smith did not state it explicitly, {the first part of} Proposition \ref{main 1} follows by the proof of \cite[Lemma 2.1]{strongly}. We shall include a proof of Proposition \ref{main 1} in Section \ref{sec loopcount} for completeness. Then, we employ this condition in a computer search to find minimal examples of regular graphs with cospectral yet non-similar vertices in Section \ref{sec example}.

\begin{theorem}\label{main 2.1}
    The smallest graph containing a pair of vertices that are cospectral but not similar has eight vertices.
\end{theorem}

Riemannian manifolds have some rigid structure that graphs, in general, lack. We impose a variety of kinds of regularity on our graphs to emulate this structure. We first search amongst the regular graphs for examples. { This setting comes with the benefit that the adjacency matrix $A$ and the two standard symmetric normalizations of the Laplacian $L = D - A$ and $\mathcal L = I - D^{-\frac 1 2} A D^{-\frac 1 2}$ all have precisely the same eigenfunctions.} Indeed, we find such examples, but only amongst graphs of ten or more vertices.

\begin{theorem}\label{main 2.2}
    The smallest regular graph containing a pair of non-similar cospectral vertices has ten vertices.
\end{theorem}

Even more structured are \emph{walk-regular graphs} \cite{godsil1980feasibility}, all of whose vertices are cospectral. A vertex $a$ has degree $(A^2)_{a,a}$, and hence all walk-regular graphs are also regular. Seeking minimal examples amongst walk-regular graphs yields:

\begin{theorem}\label{main 2.3}
    The smallest walk-regular graph containing a pair of non-similar vertices has twelve vertices.
\end{theorem}

Smooth analogs of walk-regular graphs were studied by Wang and the last two authors in \cite{echolocation2}. There they show that every compact two-dimensional surface satisfying the smooth version of walk-regularity is also homogeneous---the isometry group acts transitively on the points in the surface. The assumption that our manifold is two-dimensional is essential to the proof, and it is yet unknown if there are any higher-dimensional manifolds that are walk-regular but not homogeneous. {In addition, none of the walk-regular graphs we found are planar.} This leads us to consider planar walk-regular graphs.

We are able to prove the following discrete analog of the result of \cite{echolocation2}.

\begin{theorem}\label{main 3}
    Every 3-connected walk-regular planar graph is isomorphic to {the polyhedral graphs of} one of the following:
		\begin{itemize}
			\item 5 regular polyhedrons, also known as Platonic solids: tetrahedron, hexahedron, octahedron, dodecahedron, and icosahedron;
            \item 13 semi-regular polyhedrons, also known as Archimedean solids, pictured in Figure \ref{fig: semi-regular polyhedrons};
			\item the $m$-gonal prisms for $m = 3$ and $m \geq 5$; and
			\item the $m$-gonal antiprisms for $m \geq 4$.
		\end{itemize}
		Consequently, every walk-regular planar graph is vertex-transitive.
\end{theorem}
\begin{figure}
	\includegraphics[width=0.8\textwidth]{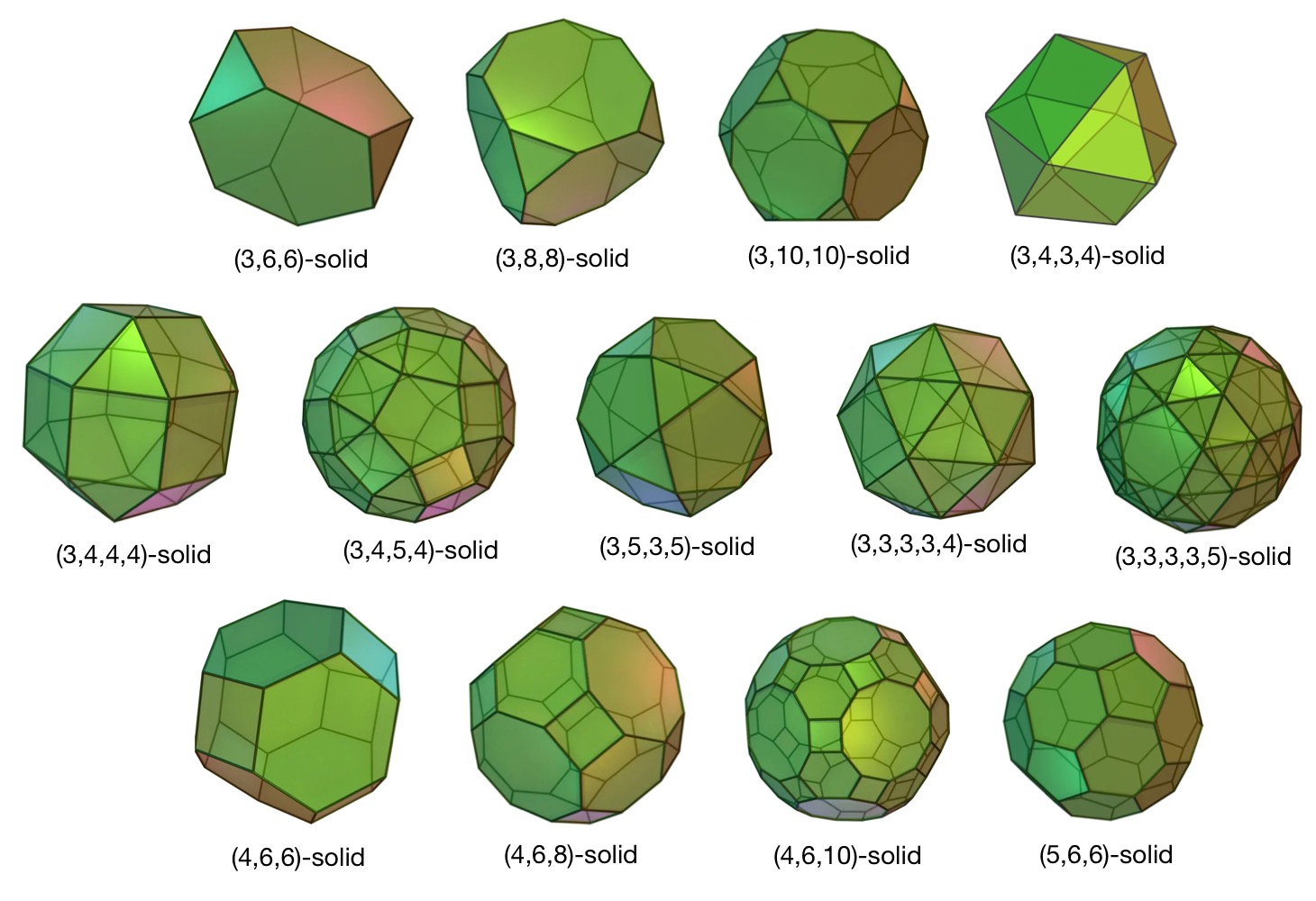}
	\caption{All the 13 semi-regular polyhedrons. The graphics are due to Wikipedia user Cyp \cite{cyp}, used with permission.}
	\label{fig: semi-regular polyhedrons}
\end{figure}

A direct consequence of the classical theorem of Mani \cite{mani} says that all 3-connected planar vertex-transitive finite graphs are exactly those listed above. Fleischner and Imrich \cite{transitiveplanar} employed the modified Schl\"afli symbol to give a direct proof of this fact.  Theorem \ref{main 3} is a refinement of this result of Mani and Fleischner--Imrich, as walk-regularity is, in general, strictly weaker than vertex-transitivity. The modified Schl\"afli symbol again plays a crucial role in our proof of Theorem 
\ref{main 3}. To be more detailed, we first attribute the counting of short closed walks to the counting of short cycles, and by some delicate arguments, we prove that in most 3-connected walk-regular planar graphs, every short cycle without chords must be a face, further demonstrating that the modified Schl\"afli symbol at each vertex is identical.

{\begin{remark}
   Although the planarity assumption might seem excessively restrictive, it is sharp in the sense that there exists a walk-regular toroidal graph that is not vertex-transitive. See Figure \ref{fig: torus}.
\end{remark}}

\subsection*{Organization of the paper}

In Section \ref{sec loopcount}, we prove Proposition \ref{main 1}. In Section \ref{sec example}, we give a short overview of our computer search and present a few of our minimal examples. The code is available in full on GitHub, along with a repository of minimal examples (https://github.com/ewy-man/graph-echolocation). Section \ref{sec planar} is dedicated to the proof of Theorem \ref{main 3}.

\subsection*{Acknowledgements} Xi was  supported by the National Key Research and Development
Program of China No. 2022YFA1007200
 and NSF China Grant No. 12171424. Kong was supported by NSF China Grant No. 12301109 and the Fundamental Research Funds for the Central Universities of China No. YJ202343. Wyman was supported by NSF grant DMS-2204397.  Wyman extends a thanks to Richard Lange for his very patient guidance through setting up a GitHub repository.

\section{Closed walks of bounded length determine if two vertices are cospectral}\label{sec loopcount}

{  Let $A$ be the adjacency matrix of some graph $G$ on $n$ vertices, or any real-symmetric $n \times n$ matrix for that matter. Let $\phi_1, \ldots, \phi_n$ be a real orthonormal eigenbasis for $A$ with respective eigenvalues $\lambda_1, \ldots, \lambda_n$. For each vertex $a$, consider the measure
\[
    \mu_a = \sum_j |\phi_j(a)|^2 \delta_{\lambda_j}
\]
on $\R$. Observe that two vertices $a$ and $b$ are cospectral if and only if $\mu_a = \mu_b$.

Let $m \leq n$ denote the count of \emph{distinct} eigenvalues and note each measure $\mu_a$ is determined by its pairing with degree $m-1$ polynomials. In particular, the measure $\mu_a$ is uniquely determined by its pairings
\[
    \int_\R t^k \, d\mu_a(t) = \sum_{j} |\phi_j(a)|^2 \lambda_j^k \qquad \text{ for } k = 0,\ldots, m-1
\]
with the standard basis $\{1, t, \ldots, t^{m-1}\}$ for the degree $m-1$ polynomials. We can arrive at the same set of quantities from a different direction, namely
\[
    A_{a,a}^k = \sum_{j,\ell} \phi_j(a) (\phi_j^t A^k \phi_\ell) \phi_\ell(a) = \sum_j |\phi_j(a)|^2 \lambda_j^k.
\]
In summary, we have shown:

\begin{lemma}
    Let $A$ be a real-symmetric $n \times n$ matrix. Then, $\mu_a = \mu_b$ if and only if
    \[
        A^k_{a,a} = A^k_{b,b} \qquad \text{ for all $k = 0,1,\ldots,n-1$.}
    \]
\end{lemma}

We note that the number $n-1$ of counts is sharp since the adjacency matrix $A$ of a simple path with $n$ vertices for every odd $n = 2\ell+1$ (i.e., $A_{i,j} = \delta_{|i-j|=1}$) as well as the adjacency matrix $A$ of a simple path with $n$ vertices and a self-loop attached to the last vertex for every even $n = 2\ell$ (i.e., $A_{i,j} = \delta_{|i-j|=1} + \delta_{i=j=n}$) satisfies
\[
    A^k_{\ell,\ell} = A^k_{\ell+1, \ell+1} \quad \text{for all}~k=0,1,\ldots,n-2 \quad \text{but} \quad A^{n-1}_{\ell,\ell} < A^{n-1}_{\ell+1,\ell+1}.
\]
}

\section{Minimal examples}\label{sec example}

An exhaustive list of minimal examples of graphs with cospectral, non-similar vertices under regularity conditions of varying strength can be found on GitHub (https://github.com/ewy-man/graph-echolocation) along with the code used to generate them.

The code is written in the Julia language (see \cite{julia, juliagraphs}) and contains a number of tools:
\begin{enumerate}
    \item A few functions that generate exhaustive lists of non-isomorphic graphs on a specified number of vertices. An optional filter may be passed to these functions to narrow down and speed up the search.
    \item A function which returns the orbits of the action of the isomorphism group on the vertices.
    \item A function that determines the cospectrality of the vertices via Proposition \ref{main 1}.
    \item A number of other convenience functions, such as those that determine if a graph is walk-regular or if there exist non-similar cospectral pairs.
\end{enumerate}
Once loaded into the Julia REPL, these functions can be called to generate the minimal examples below and to check that they are indeed minimal. The first listed operation is the most time-intensive, but typically does not take more than a minute to run on a personal computer.

\subsection{Minimal graphs with non-similar cospectral vertices} The smallest graphs that contain non-similar cospectral pairs of vertices have eight vertices. Of all 12346 non-isomorphic graphs on 8 vertices, 126 of them contain non-similar cospectral vertices. Four examples are depicted in Figure \ref{fig: minimal non-similar cospectral}. {Furthermore, all of the 126 minimal examples are connected and none of them are trees.}

\begin{figure}
    \includegraphics[width=0.7\textwidth]{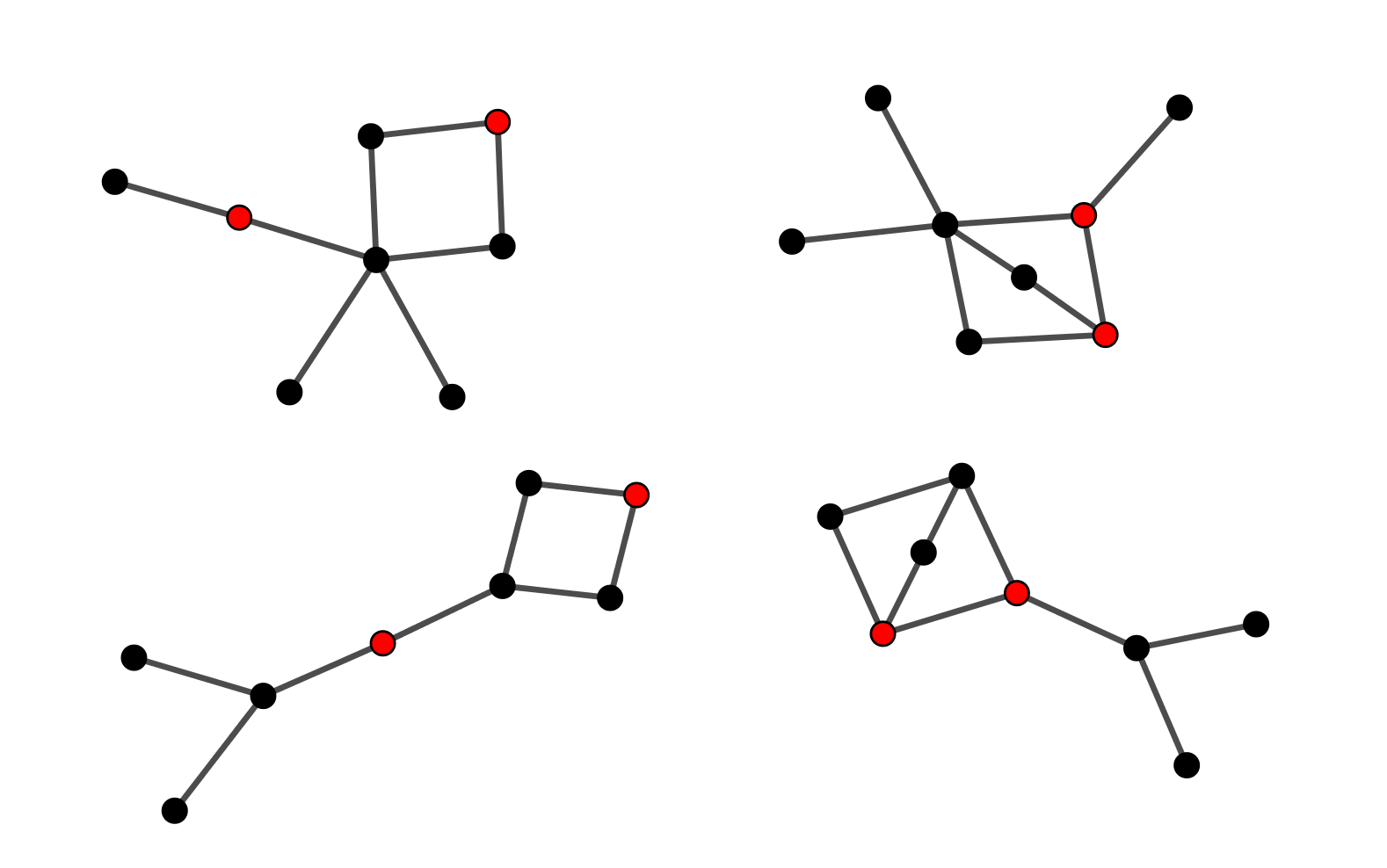}
    \caption{Four examples of graphs with eight vertices with a pair of non-similar cospectral vertices. The non-similar cospectral pairs are marked in red.}
    \label{fig: minimal non-similar cospectral}
\end{figure}

\subsection{Minimal regular graphs with non-similar cospectral vertices} The smallest regular graphs which contain non-similar cospectral pairs of vertices have ten vertices, and of those only those of degree 3, 4, 5, and 6 have minimal examples.
\begin{center}
    \begin{tabular}{c|c|c}
    $d$ & number of $d$-regular examples & number of $d$-regular graphs \\ \hline
    $3$ & 3 & 21 \\
    $4$ & 22 & 60 \\
    $5$ & 22 & 60 \\
    $6$ & 3 & 21
    \end{tabular}
\end{center}

The symmetry in the table is not a coincidence.

\begin{proposition}\label{prop: edge complements} Let $G = (V,E)$ be a finite graph and let $G' = (V, E^c)$ denote its edge complement.
\begin{enumerate}
\item $G$ and $G'$ have the same graph isomorphism group.
\item If $G$ and $G'$ are both connected, then two vertices are cospectral in $G$ if and only if they are cospectral in $G'$.
\end{enumerate}
\end{proposition}

The proposition explains the symmetry we see in the table. Indeed, all of the minimal examples counted by the table here happen to be connected. Each degree $3$ example's edge complement is a minimal degree $6$ example, and similarly for degree $4$ and $5$ examples.

{ 
\begin{remark}\label{rem: adj vs lap}
    Given a $d$-regular graph $G$ with adjacency matrix $A$, its graph Laplacian is written
    \[
        L = dI - A.
    \]
    As such, $A$ and $L$ share precisely the same eigenfunctions. Hence, two vertices $a$ and $b$ of a $d$-regular graph are cospectral if and only if
    \[
        \sum_{\lambda_j = \lambda} |\phi_j(a)|^2 = \sum_{\lambda_j = \lambda} |\phi_j(b)|^2
    \]
    where $\phi_1, \ldots, \phi_n$ is an orthonormal Laplace eigenbasis with $L\phi_j = \lambda_j \phi_j$.
\end{remark}
}

\begin{proof}
Part (1) is obvious and well-known. For part (2), note the Laplacian $L'$ on $G'$ can be written in terms of the Laplacian $L$ on $G$ as
\[
    L' = nI - \mathbf 1 - L,
\]
where $\mathbf 1$ is the $n \times n$ matrix with a one in every entry. It suffices to show that an eigenfunction $\phi$ for $L$ is also an eigenfunction for $L'$. Since $G$ and $G'$ are both connected, their trivial eigenspaces are both spanned by the constant function, so it suffices to show $L$ and $L'$ share nonconstant eigenfunctions. If $\phi$ is an eigenfunction of $L$ with eigenvalue $\lambda > 0$, it is orthogonal to the constant eigenspace and hence lies in the kernel of $\mathbf 1$. Hence,
\[
    L'\phi = (n - \lambda)\phi.
\]
The reverse argument is identical.
\end{proof}
{ We remark that, as a consequence, the edge complement of a walk-regular graph is again walk-regular. This is Theorem 4.2 in the work of Godsil and Mckay \cite{godsil1980feasibility} on walk-regular graphs.}

\subsection{Minimal walk-regular, non-vertex-transitive graphs}

There are no walk-regular, non-vertex-transitive graphs, which have less than 12 vertices. There are exactly four minimal examples with twelve vertices, one each of degree 4, 5, 6, and 7. Again, these minimal examples are connected, and so can be put into edge-complement pairs; the examples of degrees $4$ and $7$ are edge complements, and so are the examples of $5$ and $6$. The degree $4$ and $5$ examples are depicted in Figure \ref{fig: minimal example}. Furthermore, even though the degree 4 example is not planar, it can be embedded into a torus, as shown in Figure \ref{fig: torus}.

\begin{figure}
    \[
        \includegraphics[align=c,width=0.4\textwidth]{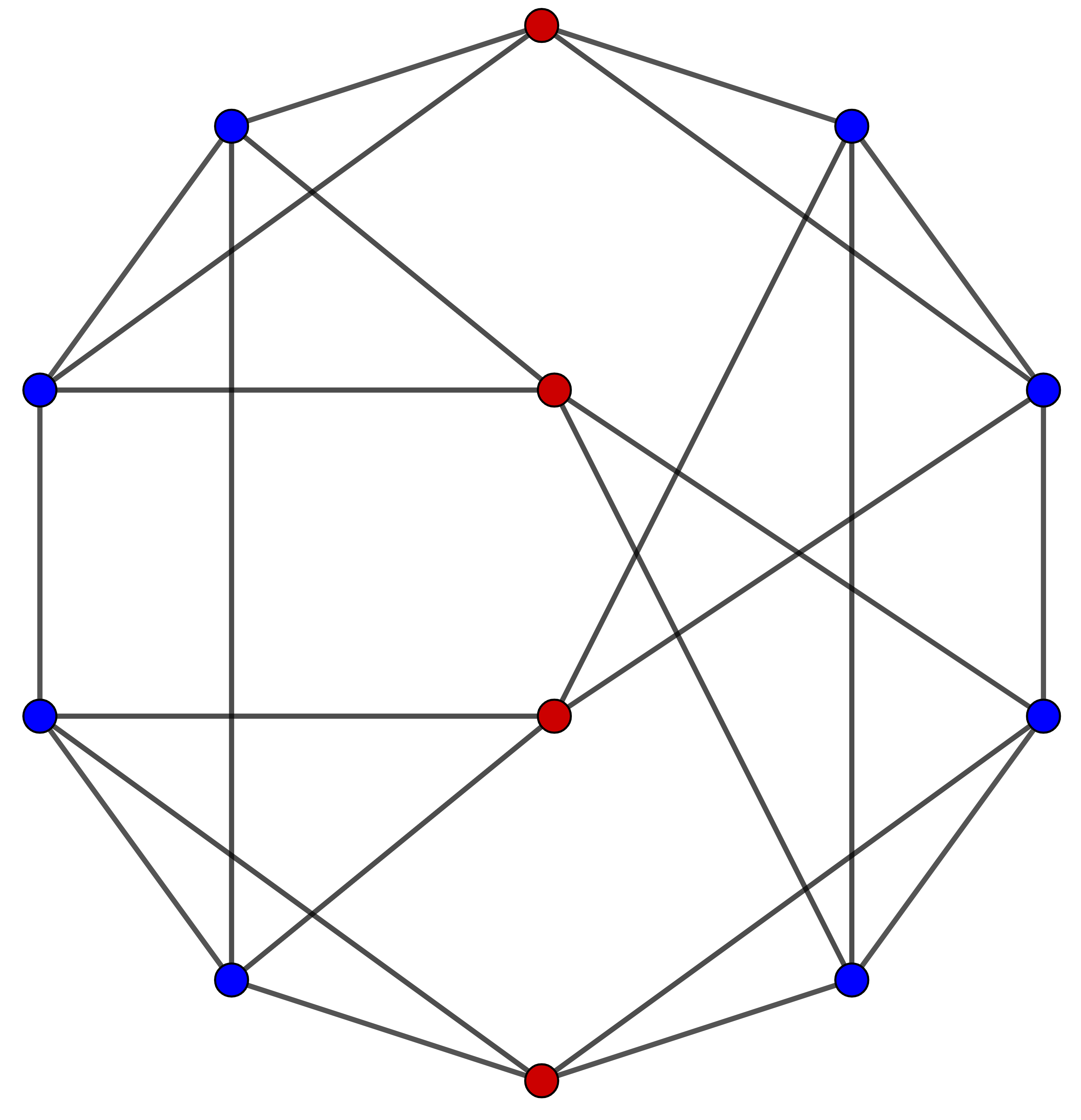}
        \qquad
        \includegraphics[align=c,width=0.4\textwidth]{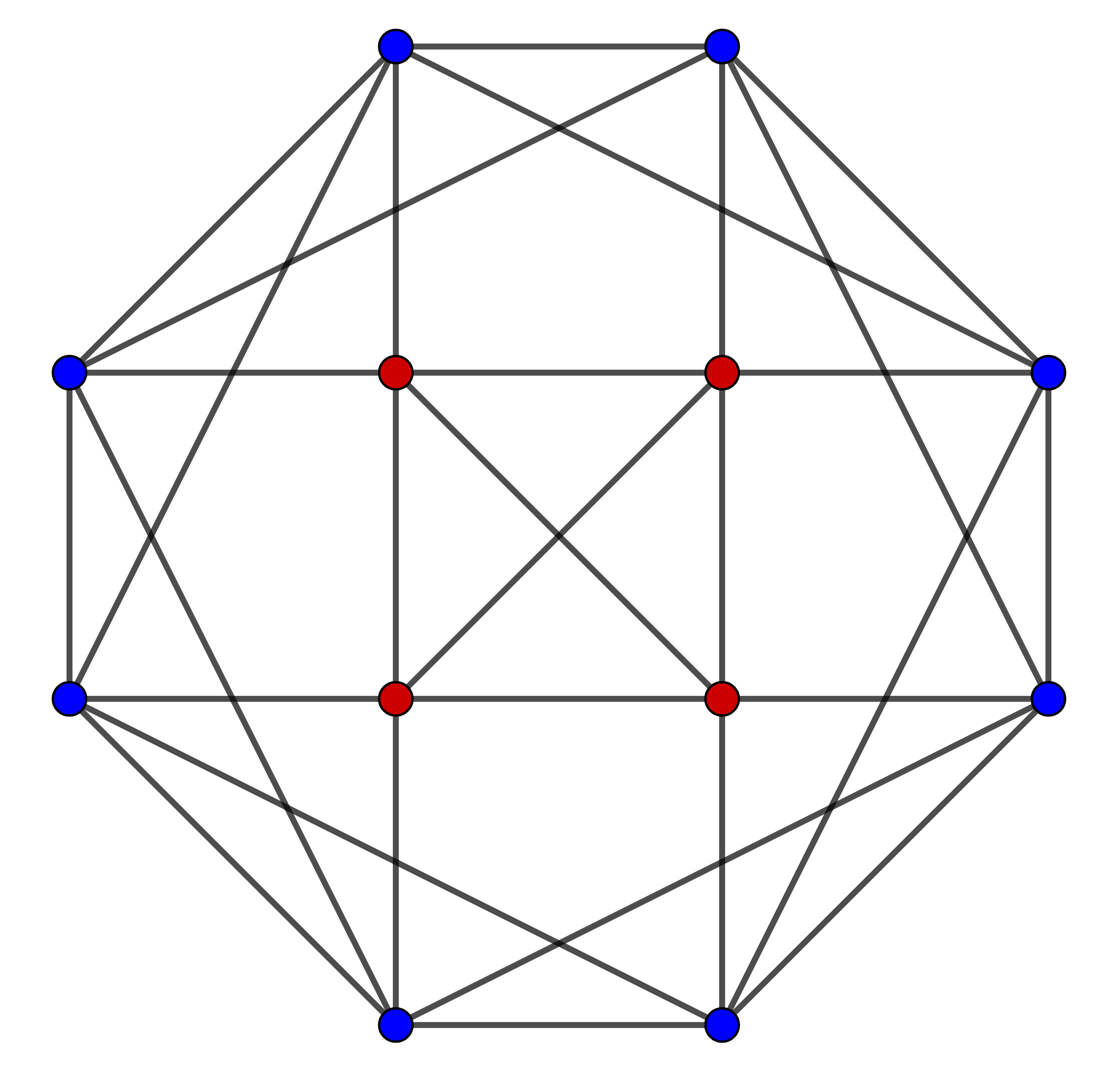}
    \]
    \caption{These graphs and their edge-complements are the only walk-regular, non-vertex-transitive graphs on $12$ vertices. Red vertices belong to one orbit and blue vertices belong to the other.} \label{fig: minimal example}
\end{figure}

\begin{figure}
    \includegraphics[width=0.45\textwidth]{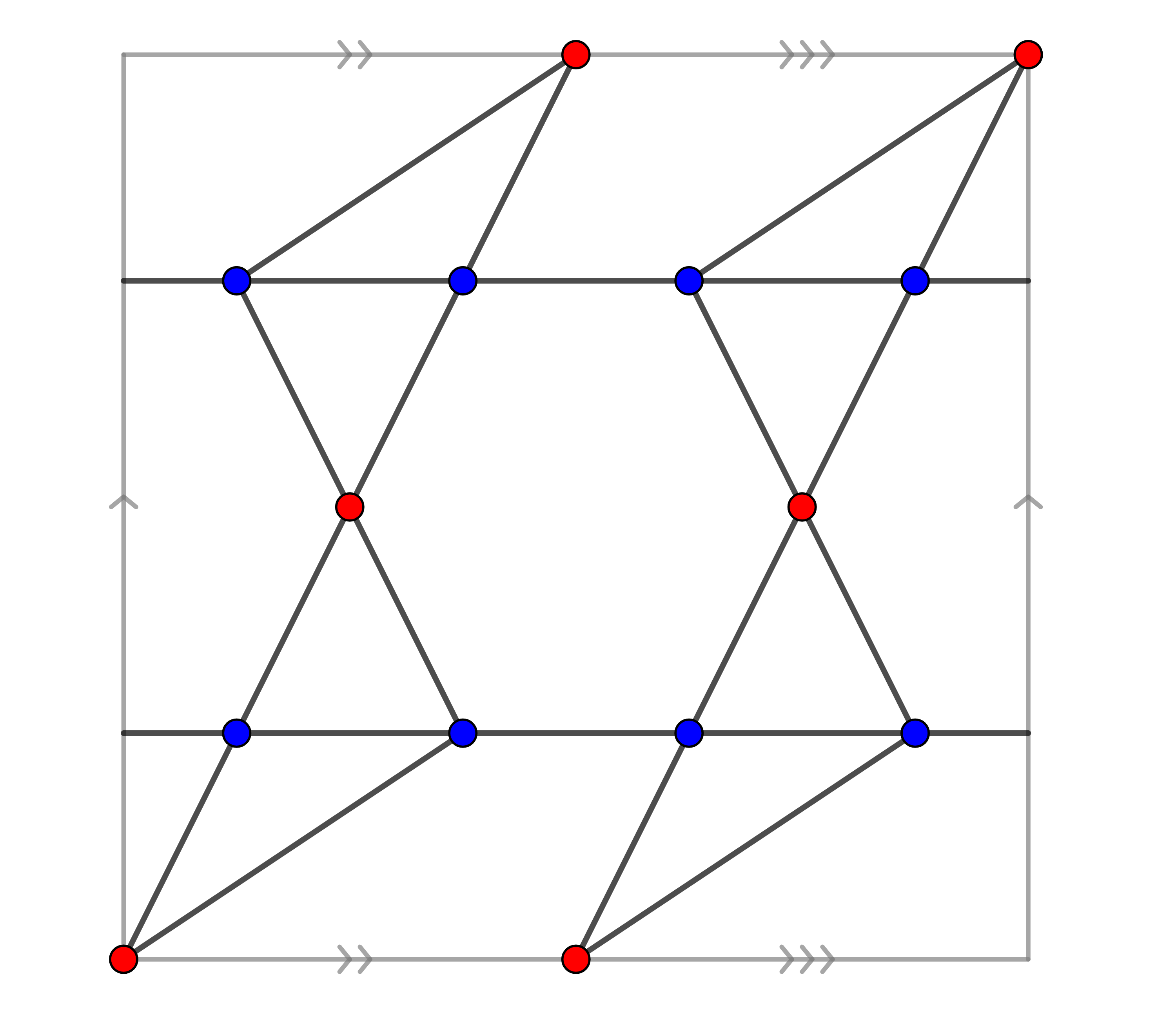}
    \caption{A walk-regular non-vertex-transitive toroidal graph.}
    \label{fig: torus}
\end{figure}

\section{Walk-regular planar graphs}\label{sec planar}

\subsection{Auxiliary lemmas}

Here we set up the tools we will need in order to classify 3-connected walk-regular planar graphs and prove Theorem \ref{main 3}.


The first tool shall be used to obtain the number of short closed walks through a vertex from the number of cycles through it. Let $G$ be a $d$-regular graph with $d \geq 3$. A closed walk is said to be \textit{contractible} if it traces no cycles. For a vertex $v$ and $m \geq 3$, we let
\[\mathcal C_m(v) = \#\{\text{$m$-cycles through $v$}\}.\] 
As every closed walk of length 3 traces a triangle in one of two directions, 
\[
(A^3)_{v,v} = 2 \mathcal C_3(v).
\]
Starting at $v$ there are exactly $d(d-1)+d^2$ contractible closed walks with length 4. Adding in the ones tracing 4-cycles, we have
\[
(A^4)_{v,v} = 2d^2 - d + 2\mathcal C_4(v).
\]
Now we assume that $\mathcal C_3(v_i) \equiv \mathcal C_3$ for all vertices $v_i$. Note that every closed walk $\overline{va_1a_2a_3a_4v}$ traces either a 5-cycle or an edge attached to a triangle; among the latter ones, there are exactly $2d\mathcal C_3$ with $v=a_2$, $2d\mathcal C_3$ with $v=a_3$, $2(d-2)\mathcal C_3$ with $a_1 = a_4$ and $v \notin \{a_2,a_3\}$, $2(d-2)\mathcal C_3$ with $a_1 = a_3$ and $a_2 \notin \{v,a_3\}$, $2(d-2)\mathcal C_3$ with $a_2 = a_4$ and $a_3 \notin\{v,a_1\}$, and $2\mathcal C_3$ with $a_1 = a_3$ and $a_2 = a_4$. In total we have
\[
(A^5)_{v,v} = 10(d-1)\mathcal C_3 + 2 \mathcal C_5(v).
\]
Furthermore, assume that $\mathcal C_4(v_i) \equiv \mathcal C_4$ for all vertices $v_i$. If $G$ is triangle-free (i.e., $\mathcal C_3 = 0$), similarly, we count the closed walks starting at $v$ with length 6: $d^3+2d^2(d-1)+2d(d-1)^2$ of them are contractible, $2 \cdot 2d\mathcal C_4 + 4 \cdot 2(d-2)\mathcal C_4 + 2\cdot 2\mathcal C_4$ of them trace some 4-cycles with attached edges, and the remaining $2\mathcal C_6(v)$ trace some 6-cycles. Therefore,
\[
(A^6)_{v,v} = 5d^3-6d^2+2d+12(d-1)\mathcal C_4 + 2\mathcal C_6(v).
\]
If $\mathcal C_3 = 1$, compared to the triangle-free case, there are two fewer contractible closed walks of length 6 but four more such walks that trace the triangle through $v$ twice. In this case, it follows that
\[
(A^6)_{v,v} = 5d^3-6d^2+2d+12(d-1)\mathcal C_4 + 2\mathcal C_6(v) + 2.
\]
A direct consequence of the above formulas is that

\begin{lemma} \label{th:cycles}
    Let $G$ be a walk-regular graph. Then
    \begin{enumerate}
        \item $\mathcal C_3(v) \equiv \mathcal C_3, \, \mathcal C_4(v) \equiv \mathcal C_4$, and $\mathcal C_5(v) \equiv \mathcal C_5$ for all vertices $v_i$;
        \item if $\mathcal C_3 \leq 1$, $\mathcal C_6(v)$ is also constant over all vertices $v_i$.
    \end{enumerate}
\end{lemma}

\medskip

Hereafter $G$ will be a 3-connected walk-regular planar graph with degree $d$ on $n$ vertices and $e$ edges. By Whitney's theorem, $G$ has an essentially unique planar embedding, where each vertex $v$ is surrounded clockwise or counterclockwise by $d$ faces $f_i(v), \, 1\leq i \leq d$. We label $v$ by its \emph{modified Schl\"afli symbol}
\[r(v) = (r_1(v),r_2(v),\ldots,r_d(v)),\]
where $r_i(v)$ is the number of edges of $f_i(v)$. The first face $f_1(v)$, as well as the labeling direction, are chosen to minimize the label $r(v)$ lexicographically so that it is uniquely defined. In particular, $3 \leq r_1(v) = \min_{1\leq i \leq d} r_i(v)$. With such a label, we count the total number $f$ of faces in the graph by
\[f = \sum_v p(v), \qquad \text{where} \ \ p(v):= \sum_{i=1}^d \frac{1}{r_i(v)}.\]
Using Euler's identity $n+f-e = 2$ and $e = nd/2$, we get
\begin{equation} \label{eq:3.1}
\frac 1n \sum_v p(v) = \frac{d-2}{2}+\frac{2}{n}.
\end{equation}

In order to prove Theorem \ref{main 3}, we shall prove the following key lemma.
\begin{lemma}\label{S symbol lemma}
    For any 3-connected walk-regular planar graph with degree $d \geq 3$, the modified Schl\"afli symbol is constant over all vertices, that is
    \begin{equation} \label{eq: constant S symbol}
    r(v) \equiv (r_1,r_2,\cdots,r_d) \qquad \textit{for all vertices}~v.
    \end{equation}
\end{lemma}

However, establishing \eqref{eq: constant S symbol} succinctly appears challenging. Instead, we will justify \eqref{eq: constant S symbol} in each individual case.

{We remark that, once we prove Lemma \ref{S symbol lemma}, it is routine to check that Theorem \ref{main 3} holds in each case. Indeed, the numbers of vertices, edges, and faces are determined by Lemma \ref{S symbol lemma} and Euler's identity. The graph can then be reconstructed in a deterministic fashion, starting with an arbitrary polygon as the outer face. We illustrate this fact for the most complicated case when $r(v)\equiv(3,4,5,4)$ in Figure \ref{fig: reconstruction of the $(3,4,5,4)$-solid.}. The only exception is the case when $r(v)\equiv(3,4,4,4)$, for which one runs into two possible reconstructions, the polyhedral graphs of the $(3,4,4,4)$-solid and twisted $(3,4,4,4)$-solid. However, the twisted one is not walk-regular. See Case (13) in the proof.}
\begin{figure}
\includegraphics[width=1\textwidth]{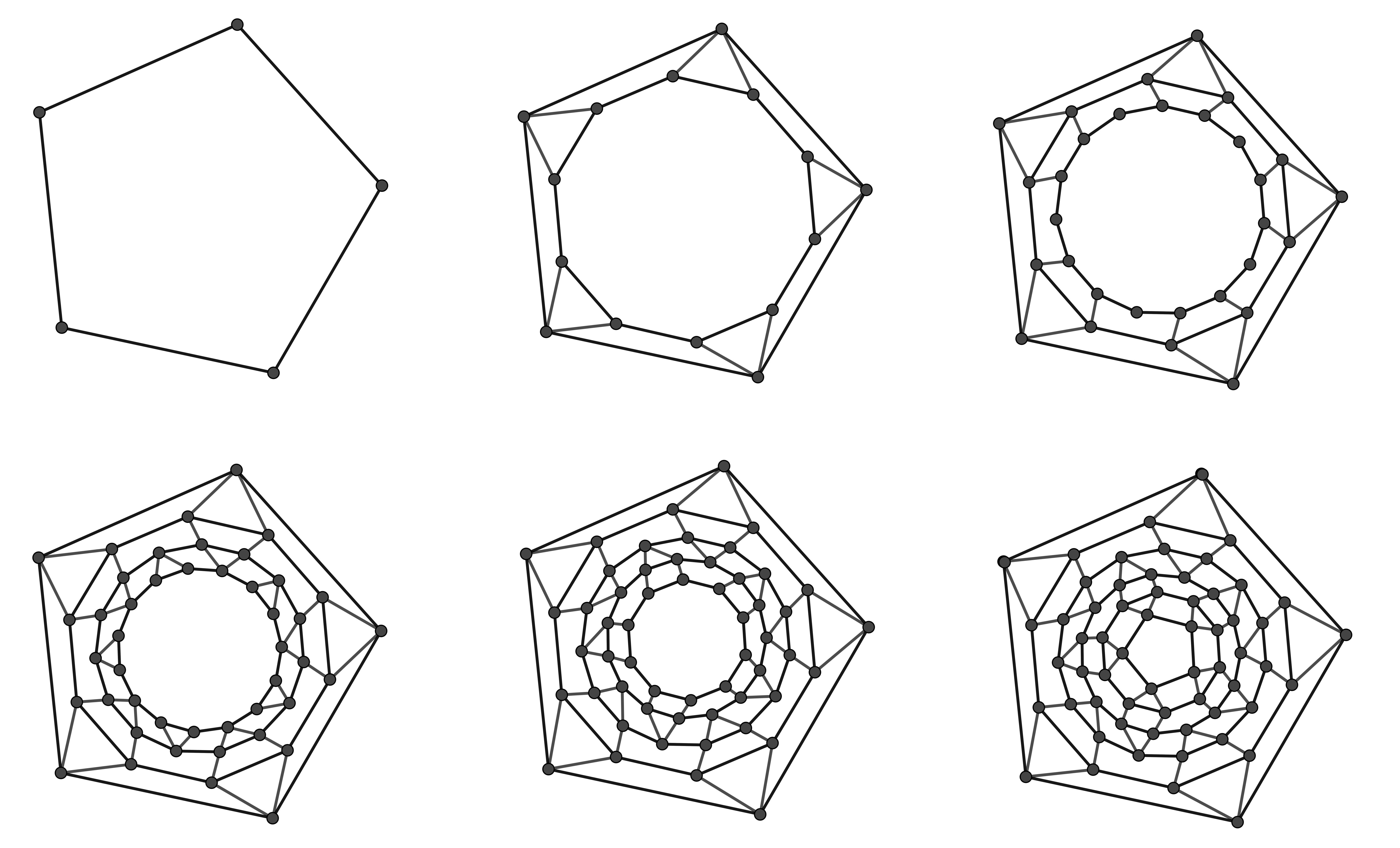}
\caption{Step by step reconstruction of the $(3,4,5,4)$-solid.}
\label{fig: reconstruction of the $(3,4,5,4)$-solid.}
\end{figure}
	
We fix a vertex $v_1$ such that $p(v_1) = \max_v p(v)$. Let 
\[r_1 = r_1(v_1) \quad \text{and} \quad k = \#\{1 \leq i \leq d: r_i(v_1) = r_1(v_1)\}.\]
In view of \eqref{eq:3.1}, to prove Lemma \ref{S symbol lemma}, we only need to discuss the fifteen cases enumerated below. For the convenience of the reader, we have also included an exhaustive list of the possible walk-regular planar graphs belonging to each case. From now on and throughout the paper, for the purpose of simplicity, we shall not distinguish between a solid and its polyhedral graph.
\begin{center}
\begin{table}
    \begin{tabular}{r|l|l}
            Case & $d, r_1, k$ & Solids  \\ \hline
            (1) & 3, 3, 3 & Tetrahedron \\
            (2) & 3, 3, 2 &  \\
            (3) & 3, 3, 1 & Triangular prism, $(3,6,6)$-, $(3, 8,8)$-, and $(3,10,10)$-solids \\
            (4) & 3, 4, 3 & Hexahedron \\
            (5) & 3, 4, 2 & $m$-gonal prism, $m \geq 5$ \\
            (6) & 3, 4, 1 & $(4,6,6)$-, $(4,6,8)$-, and $(4, 6, 10)$-solids \\
            (7) & 3, 5, 3 & Dodecahedron \\
            (8) & 3, 5, 2 &  \\
            (9) & 3, 5, 1 & $(5, 6, 6)$-solid \\
            (10) & 4, 3, 4 & Octahedron \\
            (11) & 4, 3, 3 & $m$-gonal antiprism, $m \geq 4$ \\
            (12) & 4, 3, 2 & $(3,4, 3, 4)$- and $(3, 5, 3, 5)$-solids \\
            (13) & 4, 3, 1 & $(3, 4, 4, 4)$- and $(3, 4, 5, 4)$-solids \\
            (14) & 5, 3, 5 & Icosahedron \\
            (15) & 5, 3, 4 & $(3, 3, 3, 3, 4)$- and $(3, 3, 3, 3, 5)$-solids. 
    \end{tabular}
    \caption{15 possible cases.}\label{table 1}
    \end{table}
\end{center}

\medskip

To effectively use Lemma \ref{th:cycles} in discussions, we need another tool to convert the number of short cycles into the one of faces. We label an $m$-cycle $C=\overline{a_1a_2\cdots a_ma_1}$ by its {\it inner degree sequence}
\[d_*(C) = (d_*(a_1),d_*(a_2),\cdots, d_*(a_m)),\] 
where $d_*(a_i)$ is the number of incident edges of $a_i$ inside $C$; here an object (vertex, edge, cycle, face, etc) is said to be ``inside $C$" if it lies in the region bounded by $C$ in the planar embedding. Again, $d_*(C)$ is uniquely defined, once we choose $a_1$ and the direction to minimize the sequence lexicographically. Clearly, $C$ is a face if and only if $d_*(C) = (0,0,\cdots,0)$ or $(d-2,d-2,\cdots,d-2)$ (the latter means that $C$ is the unique outer face of the embedding). For $m=3$, we have

\begin{lemma} \label{th:3-face}
    In a 3-connected walk-regular planar graph, every triangle is a face.
\end{lemma}

Note that a cycle $C$ in $G$ with length $\geq 4$ is not necessarily a face. The typical situation is that $C$ has a {\it chord}, an edge not in $C$ whose endpoints lie in $C$. For instance, the (3,3,3,3,4)-solid contains a 4-cycle with one chord, a 5-cycle with one chord and a 5-cycle with two chords; none of them are faces. We also observe that the 4-cycle encircling four triangles in the octahedron, the 5-cycle encircling five triangles in the icosahedron, and the 5-cycle encircling three triangles and a quadrangular face in the 4-gonal antiprism are not faces, even though all of them are {\it chordless}. Indeed, by excluding these cases, we have the next two lemmas for 4- and 5-cycles.

\begin{lemma} \label{th:4-face}
    In a 3-connected walk-regular planar graph which is not isomorphic to the octahedron, every chordless 4-cycle is a face.
\end{lemma}
\begin{lemma} \label{th:5-face}
    In a 3-connected walk-regular planar graph which is not isomorphic to the 4-gonal antiprism or the icosahedron, every chordless 5-cycle is a face.
\end{lemma}

Next, we will provide the proofs of Lemmas \ref{th:3-face}--\ref{th:5-face} separately in the cases $d=3,4,5$, and conclude the proof of Lemma \ref{S symbol lemma}, and thus Theorem \ref{main 3} for each case.

\subsection{The cubic cases} Assume that $G$ is cubic (i.e., $3$-regular). When a chordless $m$-cycle $C$ in $G$ is not a face, the $3$-connectedness yields that $\geq 3$ vertices on $C$ have incident edges inside $C$, and $\geq 3$ of them have incident edges outside $C$, hence $m \geq 6$. This justifies Lemmas \ref{th:3-face}--\ref{th:5-face} for all cubic cases. Consequently, for any vertex $v$, we have
\[r_1(v) \equiv r_1 \quad \text{and} \quad k(v) := \#\{1 \leq i \leq 3: r_i(v) = r_1\} \equiv k.\]

Among the cases (1)--(9) in Table \ref{table 1}, (2) and (8) are simply impossible to realize. Also, one quickly sees the symbols $r(v)$ of the graphs realizing cases (1), (4), and (7) must be constant, where the corresponding determined solids are the tetrahedron, the hexahedron, and the dodecahedron. 

\medskip

In case (5) ($d = 3, r_1 = 4, k = 2$), the face $f_3(v_1)$ surrounded by $m = r_3(v_1) \geq 5$ quadrangular faces determines the unique pattern up to isomorphism. Therefore, $r(v) \equiv (4,4,m)$ and the graph is isomorphic to the $m$-gonal prism.

\medskip

For the other three cases (with $d=3,k=1$), we need to consider the second shortest cycles in $G$, that is, the cycles with length
\[\ell = \min\{m > r_1: \mathcal C_m(v) > 0 \text{ for some vertex } v\}.\]
It is clear that $\ell \leq r_2(v_1)$. We will claim in each individual case that 
\begin{lemma} \label{th:2nd shortest cycle}
    In a 3-connected walk-regular planar graph with $d=3$ and $k=1$, every $\ell$-cycle is a face.
\end{lemma}

\subsection*{Case (3)} ($d=3,r_1=3,k=1$) Note that each $\ell$-cycle is chordless, and it should be adjacent to exactly $\ell/2$ triangles, and thus $\ell$ must be even. Suppose that an $\ell$-cycle $C$ is not a face. Then by the 3-connectedness, $C$ has $\geq 3$ neighbors inside $C$ and $\geq 3$ neighbors outside $C$. Since $C$ is chordless and each vertex must be adjacent to a triangle,  we must have $\ell \geq 12$. However, in this situation one has $r_2(v_1) \geq 12$ and 
\[p(v_1) \leq \frac 13 + \frac 1{12} + \frac 1{12} = \frac 12,\]
which violates \eqref{eq:3.1}. This proves Lemma \ref{th:2nd shortest cycle} for $r_1 = 3$, and the possible $\ell = 4,6,8$ or $10$. 

Assume $\ell = 4$ or $6$. It follows from Lemma \ref{th:cycles} that $\mathcal C_4(v)$ or $\mathcal C_6(v)$ is constant for all vertices $v$. For $\ell = 8$, we can count the number of $8$-cycles though a vertex in its 4-neighborhood to obtain 
\[
(A^8)_{v,v} = 591 + 2 \mathcal C_8(v).
\]
For $\ell = 10$, a similar counting in the 5-neighborhood of the vertex $v$ provides 
\[
(A^{10})_{v,v} = 4223 + 2 \mathcal C_{10}(v).
\]
Hence whatever $\ell$ is, we have $\mathcal C_\ell(v) \equiv \mathcal C_\ell$. Since 3 is odd, drawing a triangular face on which each vertex is adjacent to exactly one face with length $\ell$ is impossible. The only other possibility is $\mathcal C_\ell = 2$, and it follows that $r(v) \equiv (3,\ell,\ell)$. The corresponding solids are the triangular prism, the $(3,6,6)$-solid, the $(3,8,8)$-solid and the $(3,10,10)$-solid respectively.

\subsection*{Case (6)} ($d=3,r_1=4,k=1$) We first claim that every cycle has even length. If this was false, we let $\ell' = \min\{\text{odd } m: \mathcal C_m(v)> 0 \text{ for some vertex } v\}$, and let $C$ be a {\it minimal} $\ell'$-cycle, in the sense that no $\ell'$-cycles are properly contained in the closed region bounded by $C$. 
Suppose that a vertex $b_1$ is inside $C$. The walk-regularity implies that, starting at $b_1$ there is at least one closed walk with length $\ell'$, denoted by $C' = \overline{b_1b_2 \cdots b_
{\ell'}b_1}$. By the minimality of $\ell'$ and $C$, $C'$ should be a cycle, and some of vertices on $C'$ lie outside $C$. Let $b_i, b_j$ be the first and the last intersection vertices of $C$ and $C'$ respectively, that is, $i = \min \{m: b_m~\text{lies on}~C\}$ and $j = \max \{m: b_m~\text{lies on}~C\}$. Such $b_i$ and $b_j$ divide $C$ into two paths $\mathcal P_1, \mathcal P_2$ with length $\ell_1, \ell_2$, and since $\ell_1 + \ell_2 = \ell'$ is odd, one of $\ell_1,\ell_2$ (say $\ell_1$) has the same parity as $\ell'+i-j$. As the union of two paths $\overline{b_ib_{i-1} \cdots b_1 b_{\ell'} b_{\ell'-1} \cdots b_j}$ and $\mathcal P_2$ form a cycle with odd length $\ell'+i-j + \ell_2$ inside the closed region bounded by $C$, it follows again from the minimality of $\ell'$ and $C$ that $\ell'+i-j + \ell_2 > \ell'$. But now, one has the union of two paths $\overline{b_ib_{i+1}\cdots b_j}$ and $\mathcal P_1$ as a closed walk with odd length $j-i+\ell_1 < \ell'$, which still violates the minimality of $\ell'$. 
Therefore, $C$ has no inner neighbors, and it should be an inner face in the planar embedding; however, this is also impossible, since at least one adjacent 4-cycle of $C$ shares three vertices with $C$, of which the middle one must have an inner incident edge contradicting Lemma \ref{th:4-face}. The claim is justified.

Using \eqref{eq:3.1}, we obtain $\ell \leq r_2(v_1) < 8$. By the claim, the only possibility is $\ell = r_2(v_1) = 6$. For a 6-cycle $C$, it has either two or three adjacent $4$-cycles. In the former situation, $C$ has only two inner neighbors and two outer neighbors, which violates 3-connectedness. For the latter one, to be 3-connected all of the three adjacent $4$-cycles must lie outside $C$ whenever $C$ is not the outer face. This proves that every $6$-cycle is a face, and justifies Lemma \ref{th:2nd shortest cycle} for $r_1 = 4$.

In view of Lemma \ref{th:cycles}, $\mathcal C_6(v) \equiv \mathcal C_6 = 1$ or $2$. For $\mathcal C_6 = 2$, we have $r(v) \equiv (4,6,6)$, and the graph is isomorphic to the $(4,6,6)$-solid. 

Now we consider the sub-case $\mathcal C_6 = 1$. It follows from \eqref{eq:3.1} and the claim that $r_3(v_1) = 8$ or $10$. For a vertex $v$ with $r(v) = (4,6,m)$, by counting in the 4- and 5-neighborhood of $v$ we obtain
\[
(A^8)_{v,v} = 809 + \begin{cases}
    \, 2 \quad \text{if } m = r_3(v_1) = 8, \\
    \, 0 \quad \text{if } m > r_3(v_1) = 8,
\end{cases}
\]
and
\[
(A^{10})_{v,v} = 6063 + \begin{cases}
    \, 2 \quad \text{if } m = r_3(v_1) = 10, \\
    \, 0 \quad \text{if } m > r_3(v_1) = 10.
\end{cases}
\]
Hence, $r(v) \equiv (4,6,8)$ or $\equiv (4,6,10)$, which corresponds to the $(4,6,8)$-solid or the $(4,6,10)$-solid respectively.

\subsection*{Case (9)} ($d=3,r_1=5,k=1$) Using \eqref{eq:3.1} we immediately get $\ell = r_2(v_1) = 6$. We will prove that every $6$-cycle is a face in this case, and conclude the proof of Lemma \ref{th:2nd shortest cycle}.

\begin{proof}[Proof of Lemma \ref{th:2nd shortest cycle} for $r_1 = 5$]
   Suppose that a 6-cycle $C = \overline{a_1a_2 \cdots a_6a_1}$ is not a face. Without loss of generality, we further assume that such $C$ is {\it minimal} among all the 6-cycles that are not faces. More precisely, this minimality of $C$ implies that any 6-cycle properly contained in the closed region bounded by $C$ should be a face. 

   By the 3-connectedness of the graph, $C$ has exactly three inner incident edges and three outer incident edges. There is a unique (up to isomorphism) possible neighborhood of $C$ such that $\mathcal C_5(a_i) = 1$ for all vertices $a_i$ on $C$, as depicted in Figure \ref{fig: 6-cycle-1}, where $C$ is adjacent to a face $\overline{b_2a_4a_5a_6b_3b_2}$ with length 5 inside $C$ and a pentagon face outside $C$ bounded by the path $\overline{a_1a_2a_3}$. Note that the inner neighbor $b_1$ of $a_2$ is not adjacent to $b_2$ or $b_3$, and hence, both the face bounded by the path $\overline{b_1a_2a_1a_6b_3}$ and the one bounded by $\overline{b_1a_2a_3a_4b_2}$ have length $>5$. Besides these two faces, the remaining adjacent face of $b_1$ should be of length $5$ so that $\mathcal C_5(b_1) = 1$. However, no matter which adjacent face of $b_1$ is of length $6$, the minimality of $C$ implies that $\mathcal C_6(b_1) < \mathcal C_6(a_2)$, which contradicts Lemma \ref{th:cycles}.
    \begin{figure}
    \includegraphics[width=0.5\textwidth]{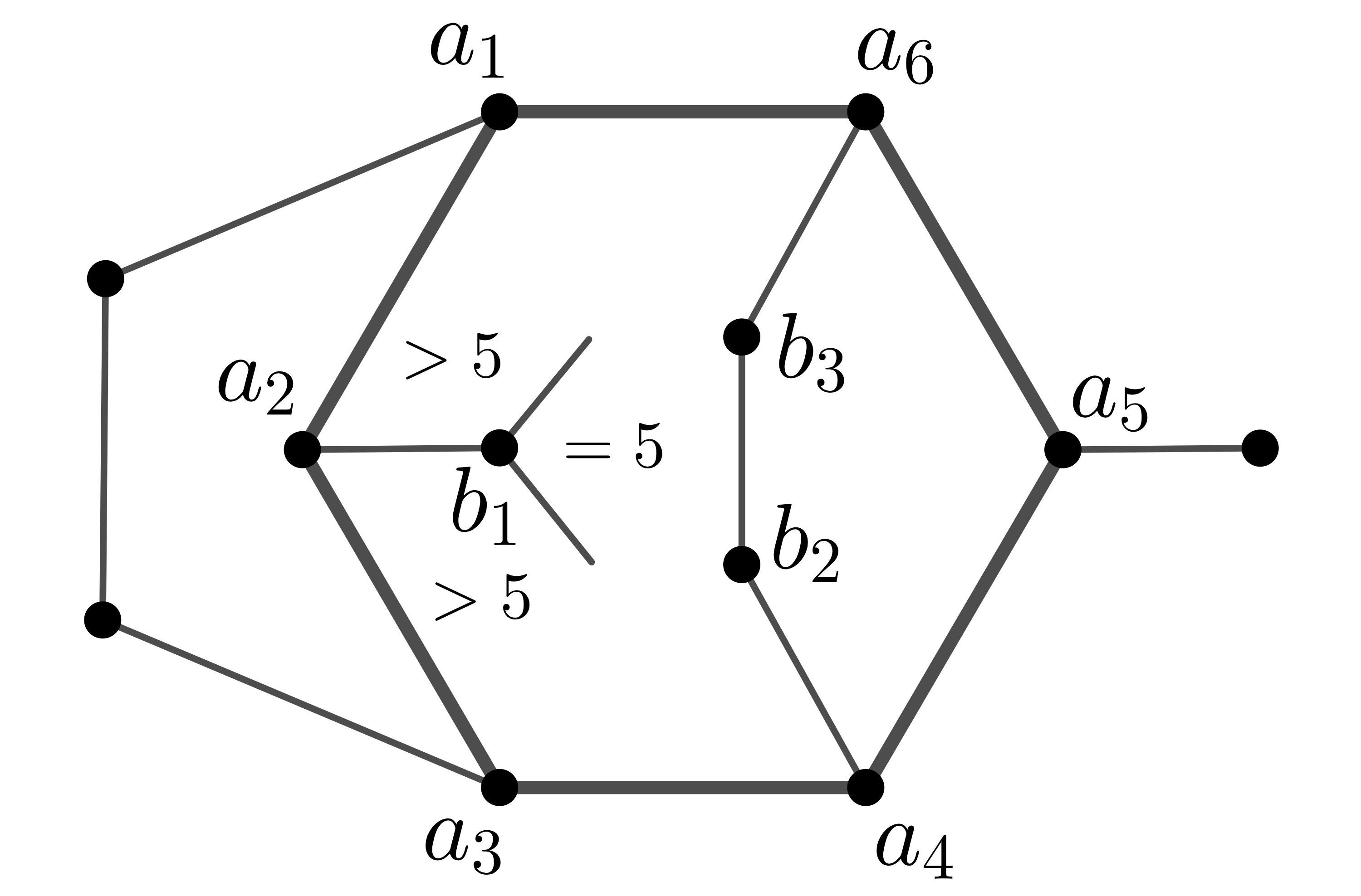}
    \caption{An impossible 6-cycle with $r_1 = 5$.}
    \label{fig: 6-cycle-1}
    \end{figure}
\end{proof}

Now using Lemmas \ref{th:cycles} and \ref{th:2nd shortest cycle} we have $\mathcal C_6(v) \equiv \mathcal C_6 = 1$ or $2$. A pentagon face on which each vertex is adjacent to exactly one face with length 6 is simply impossible, and hence $\mathcal C_6 = 2$. This yields that $r(v) \equiv (5,6,6)$, and the graph is isomorphic to the $(5,6,6)$-solid. 

\subsection{The 4-regular cases} Now we assume that $G$ is 4-regular. 
We first justify that every triangle in $G$ is a face:

\begin{proof}[Proof of Lemma \ref{th:3-face} for $d=4$]
Suppose that a triangle $C$ is not a face. By the 3-connectedness, each vertex $a_i$ has at least one incident edge inside $C$ and at least one incident edge outside $C$. Therefore, it must have $d_*(C) = (1,1,1)$. However, this is simply impossible as the total degree of the induced subgraph of the vertices inside $C$ is odd.
\end{proof}

Consequently, the octahedron is determined as $r(v) \equiv (3,3,3,3)$ in case (10). For case (11) ($d=4,r_1=3,k=3$), the face $f_4(v_1)$ surrounded by $m = r_4(v_1) \geq 4$ triangular faces determines the unique pattern up to isomorphism, hence $r(v) \equiv (3,3,3,m)$ and the graph is isomorphic to the $m$-gonal antiprism.

\medskip

The other two cases (12) and (13) are more involved. We need to prove that every 4-cycle is either a face or the boundary of two adjacent triangular faces there. 

\begin{proof}[Proof of Lemma \ref{th:4-face} for $d=4$]
    Suppose that a chordless 4-cycle $C = \overline{a_1a_2a_3a_4a_1}$ is not a face. By the 3-connectedness, at least three of $\{a_1,a_2,a_3,a_4\}$ have incident edges inside $C$, and at least three of them have incident edges outside $C$. Without loss of generality, we assume that such $C$ is {\it minimal}, that is, any chordless 4-cycle properly contained in the closed region bounded by $C$ should be a face. 
 
    Since the total degree of the subgraph inside $C$ is even, the number of incident edges of $\{a_1,a_2,a_3,a_4\}$ inside $C$ equals $4$. The possible inner degree sequence of $C$ is $(0,1,1,2), (0,1,2,1)$ or $(1,1,1,1)$. The case $d_*(C) = (0,1,2,1)$ is simply ruled out, as $\{a_2,a_4\}$ should not be a cut-set of the graph. 

    \medskip

    \noindent {\it Case: $d_*(C) = (0,1,1,2)$.} Note that $\mathcal C_3 = \mathcal C_3(v_1) \geq 1$. In this case,  $\mathcal C_3 \geq 2$ is impossible since $\{a_1,a_2,a_3,a_4\}$ has $\geq 3$ neighbors inside $C$ and $\geq 3$ neighbors outside $C$ (by the 3-connectedness). It follows that $\mathcal C_3 = 1$, and every 4-cycle inside $C$ is chordless so is a face. Using \eqref{eq:3.1}, we have $\mathcal C_4 = \mathcal C_4(v_1) \geq 2$. Lemma \ref{th:3-face} implies that every vertex inside $C$ is surrounded by at most one face with length $>4$. Let $b_1,b_2$ be the two neighbors of $a_4$ inside $C$. When $\overline{b_1b_2}$ is an edge, to have $\mathcal C_3(a_3) = \mathcal C_3(a_4) = 1$, $a_2$ and $a_3$ have exactly one common neighbor $u$, and this gives a cut-set $\{a_4,u\}$ if $u$ lies inside $C$ or a cut-set $\{a_1,u\}$ if $u$ lies outside $C$, which violates the 3-connectedness.   
    
    Now the only possibility left is that one of $b_1,b_2$ (say $b_1$) is adjacent to $a_3$, and one of the neighbors of $a_1$ outside $C$ is adjacent to $a_2$. Let $b_3$ be the neighbor of $a_2$ inside $C$. Since the face bounded by $\overline{b_2a_4a_1a_2b_3}$ has length $>4$, we have $\mathcal C_4 = 2$, and both the face bounded by $\overline{b_2a_4b_1}$ and the one bounded by $\overline{b_3a_2a_3b_1}$ have length $4$, as drawn in Figure \ref{fig: 4-cycle-0-1-1-2}. Then, to have $\mathcal C_3(b_1) + \mathcal C_4(b_1) = 3$, the undetermined face with $\overline{b_1b_3}$ should have length $>4$, contradicting the fact that $\mathcal C_3(b_3) + \mathcal C_4(b_3) = 3$. 
    \begin{figure}
    \includegraphics[width=0.5\textwidth]{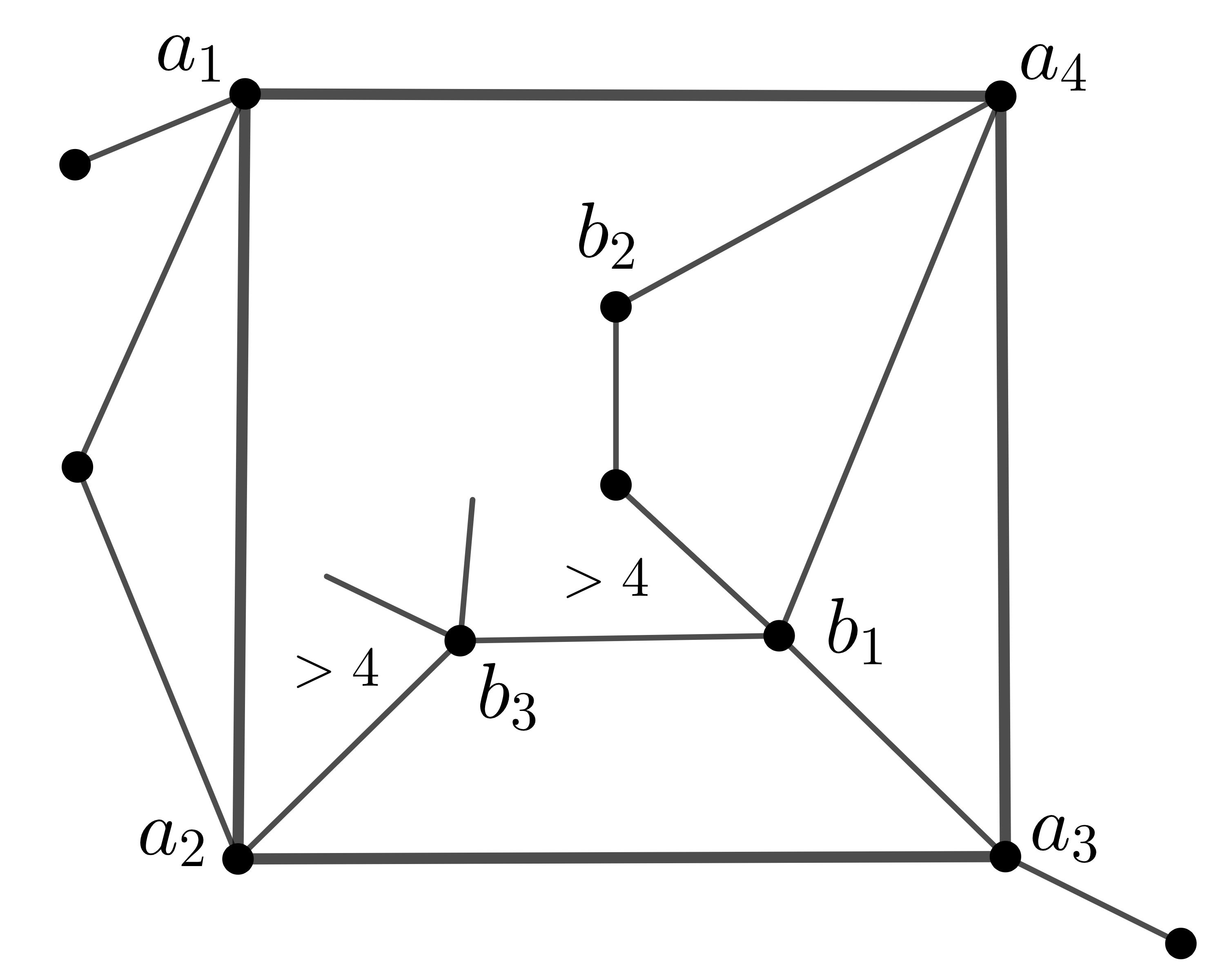}
    \caption{An impossible 4-cycle with $d_*(C) = (0,1,1,2)$.}
    \label{fig: 4-cycle-0-1-1-2}
    \end{figure}

    \medskip

    \noindent {\it Case: $d_*(C) = (1,1,1,1)$.} One special situation is that $\{a_1,a_2,a_3,a_4\}$ has a common neighbor inside $C$, in which $\mathcal C_3 = 4$ and the graph is isomorphic to the octahedron. If this is not the case, again by the 3-connectedness, there are at least three distinct neighbors of $\{a_1,a_2,a_3,a_4\}$ inside $C$ as well as at least three distinct ones outside $C$. From this we know that $\mathcal C_3 = 1$, one edge (say $\overline{a_1a_4}$) of $C$ lies on an incident triangle inside $C$, and the opposite edge ($\overline{a_2a_3}$) of $C$ lies on an incident triangle outside $C$. Denote by $b_1,b_2,b_3$ the neighbors of $a_1,a_2,a_3$ inside $C$ respectively. Note that there are no chorded $4$-cycles inside $C$ or else $\mathcal C_3 \geq 2$, and hence every $4$-cycle inside $C$ is a face by the minimality of $C$, thus $\mathcal C_4 \leq 3$.
    
    If $\mathcal C_4 = 3$, then none of the faces inside $C$ has length $>4$, and it forces $\overline{b_1b_2b_3b_1}$ to be a triangle but not a face, which violates Lemma \ref{th:3-face}.  
    
    Now we suppose $\mathcal C_4 = 2$. Once $b_2,b_3$ are adjacent, to have $\mathcal C_4(a_2) = \mathcal C_4(a_3) = 2$, both the face bounded by $\overline{b_2a_2a_1b_1}$ and the one bounded by $\overline{b_3a_3a_4b_1}$ have length $>4$, which yields that $\mathcal C_4(b_1) \leq 1$. Hence any graph with $\overline{b_2b_3}$ is ruled out.   
    
    Without $\overline{b_2b_3}$, the face bounded by $\overline{b_2a_2a_3b_3}$ has length $>4$. To have $\mathcal C_4(b_2) = \mathcal C_4(b_3) = 2$, both $b_2,b_3$ must be adjacent to $b_1$, which gives a cut-set $\{b_2,b_3\}$ of the graph (Figure \ref{fig: 4-cycle-1-1-1-1}), a contradiction.

    \begin{figure}
    \includegraphics[width=0.5\textwidth]{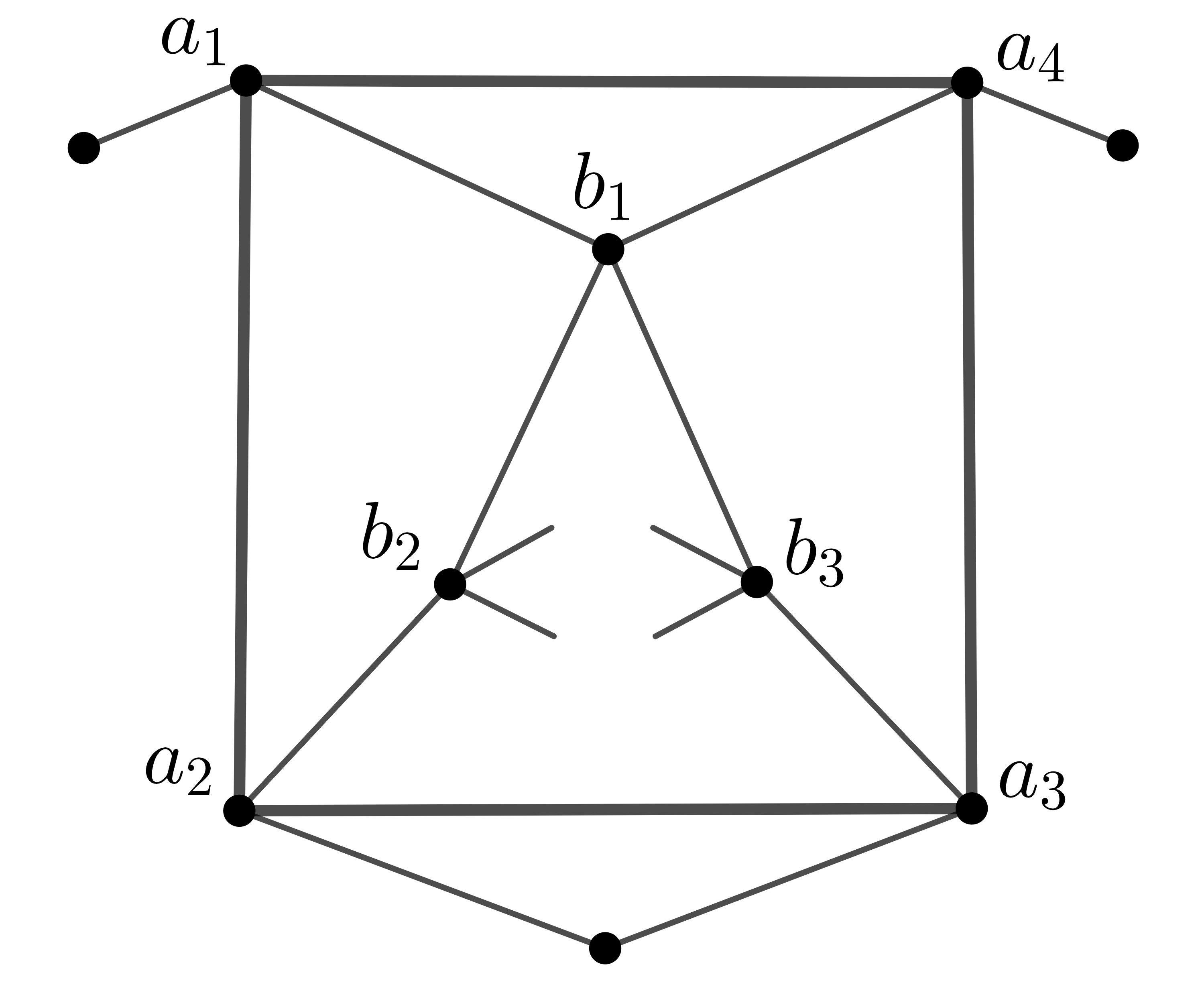}
    \caption{An impossible 4-cycle with $d_*(C) = (1,1,1,1)$.}
    \label{fig: 4-cycle-1-1-1-1}
    \end{figure}
\end{proof}

\subsection*{Case (12)} ($d=4, r_1=3, k=2$) We first justify Lemma \ref{th:5-face} for this case.

\begin{proof}[Proof of Lemma \ref{th:5-face} for case (12)]
    Suppose a chordless 5-cycle $C = \overline{a_1a_2a_3a_4a_5a_1}$ is not a face. 
    Note that the number of inner incident edges of $C$ must be even. By the 3-connectedness, the possible choices for $d_*(C)$ are $(0,0,1,1,2)$, $(0,1,0,1,2)$, $(0,1,1,1,1)$, $(0,1,1,2,2)$, $(0,1,2,1,2)$ or $(1,1,1,1,2)$. In view of the symmetry between the inside and outside of $C$, it suffices to rule out the first three situations, as illustrated in Figure \ref{fig: 5-cycle}.
    \begin{figure}
    \includegraphics[width=0.28\textwidth]{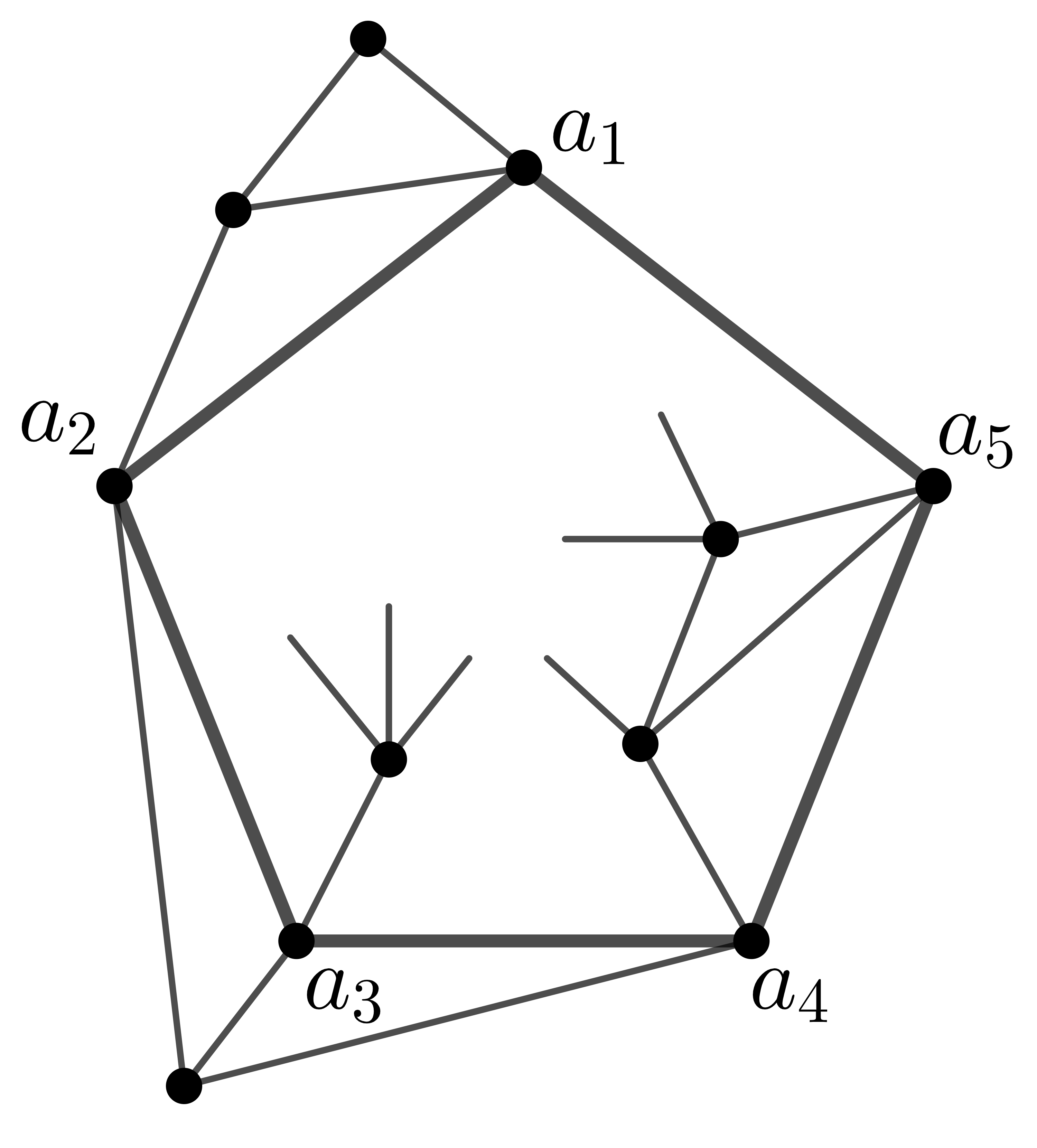}
    \includegraphics[width=0.3\textwidth]{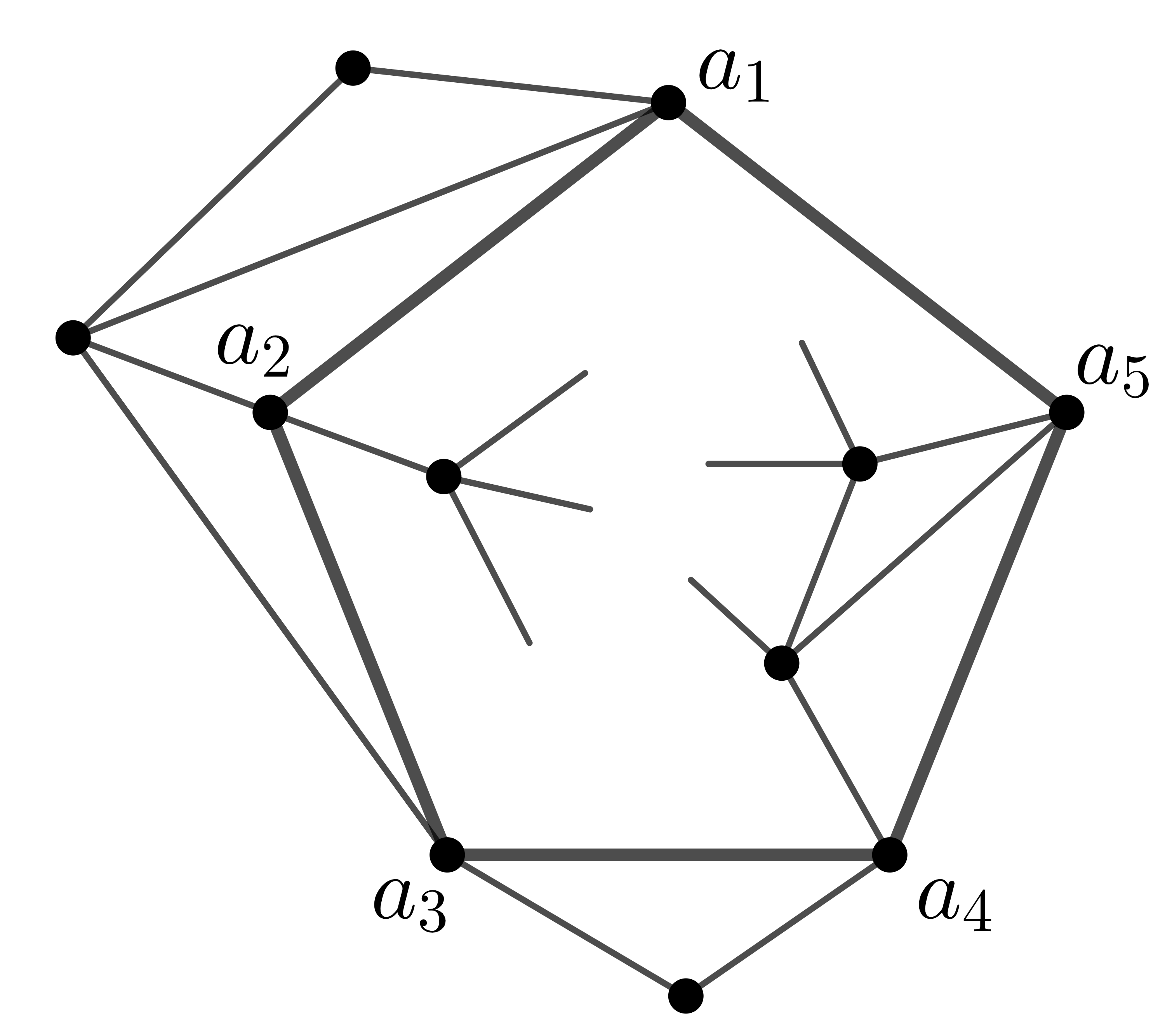}
    \includegraphics[width=0.33\textwidth]{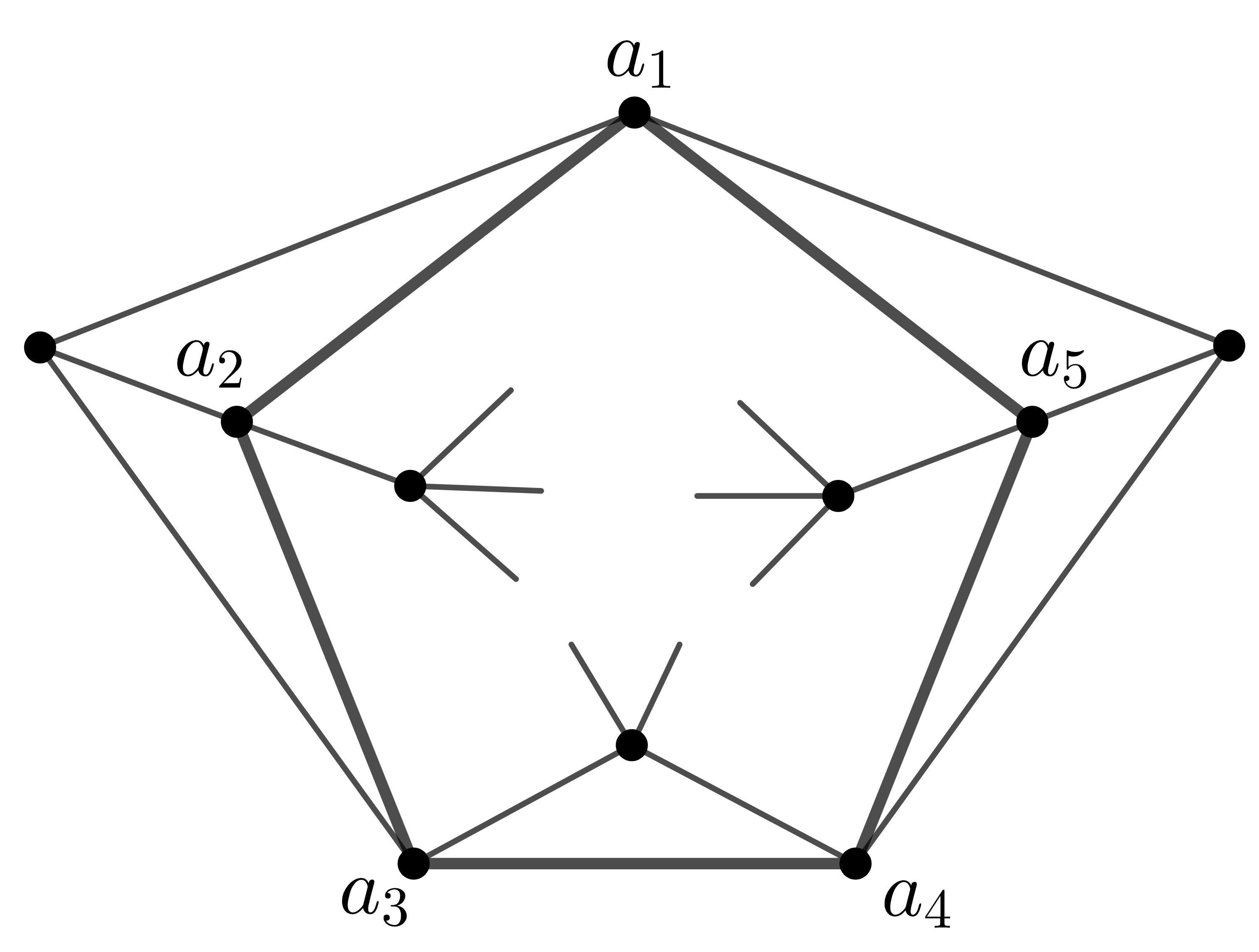}
    \caption{Three impossible 5-cycles with $d_*(C) = (0,0,1,1,2)$, $(0,1,0,1,2)$ and $(0,1,1,1,1)$ respectively.}
    \label{fig: 5-cycle}
    \end{figure}

    \medskip
    
    \noindent {\it Sub-case: $d_*(C) = (0,0,1,1,2)$.} Note that one of the two incident triangles of $a_1$ is bounded by $\overline{a_1a_2}$, and one of the two incident triangles of $a_5$ is bounded by $\overline{a_4a_5}$. To have three distinct inner neighbors of $C$ (by the 3-connectedness), $a_3$ has no incident triangles inside $C$, and thus $a_2,a_3,a_4$ have a common outer neighbor. This leads to $\mathcal C_4(a_4) \geq 2 > \mathcal C_4(a_5)$, a contradiction.

    \medskip

    \noindent {\it Sub-case: $d_*(C) = (0,1,0,1,2)$.} As in the first sub-case, a triangle outside $C$ is bounded by $\overline{a_1a_2}$ and one inside $C$ is bounded by $\overline{a_4a_5}$. It follows similarly that $a_2$ has no inner incident triangles, and this forces $a_1,a_2,a_3$ to have a common neighbor outside $C$ which has three incident triangles. 

    \medskip

    \noindent {\it Sub-case: $d_*(C) = (0,1,1,1,1)$.} By the 3-connectedness, there is at most one inner incident triangle of $C$. We also note that, in $\{a_2,a_3,a_4,a_5\}$ the vertices with two incident triangles outside $C$ cannot be adjacent, otherwise four of the vertices on $C$ have a common neighbor that is incident to three triangles. Therefore, the only possible choice of the unique inner incident triangle of $C$ is the one bounded by $\overline{a_3a_4}$. However, in this situation $C$ has only two outer neighbors, which violates the 3-connectedness.
\end{proof}

Now we are ready to finish off Case (12). In this case, each vertex is incident to two triangles and two other faces of length $\geq 4$. In view of Lemma \ref{th:cycles}, we subdivide the argument into four sub-cases where $\mathcal C_4(v) \equiv \mathcal C_4 = 0,1,2$ or $\geq 3$.

\subsubsection*{Sub-case: $\mathcal C_4 = 0$} There are no faces of length $4$, and no two triangles share an edge. It follows that for each $v$, $r(v)$ is of the form $(3,m,3,m')$ for some $5 \leq m \leq m'$.
Such graphs have plenty of structure. Firstly, each face of length $\geq 5$ is adjacent only to triangular faces. Furthermore, each such face is, through each vertex, opposed to another face of length $\geq 5$. 

Using Lemma \ref{th:cycles} and \eqref{eq:3.1}, we have $\mathcal C_5(v) = \mathcal C_5(v_1) \geq 1$ for any vertex $v$. Also we note that every 5-cycle is chordless, and so is a face by Lemma \ref{th:5-face}. If we attempt to construct a triangle on which each vertex is incident to exactly one pentagon face, we quickly run into absurdities. Hence, we are left with $r(v) \equiv (3,5,3,5)$, and the graph is isomorphic to the $(3,5,3,5)$-solid.

\subsubsection*{Sub-case: $\mathcal C_4 = 1$} First, we rule out the possibility that such a graph has a square face. If we start with a square face, we obtain a graph resembling that in Figure \ref{fig: g3}, which we will call a ``bee" since it looks like it has a head with two antennae, a thorax with two wings, and an abdomen.
\begin{figure}
\includegraphics[width=0.45\textwidth]{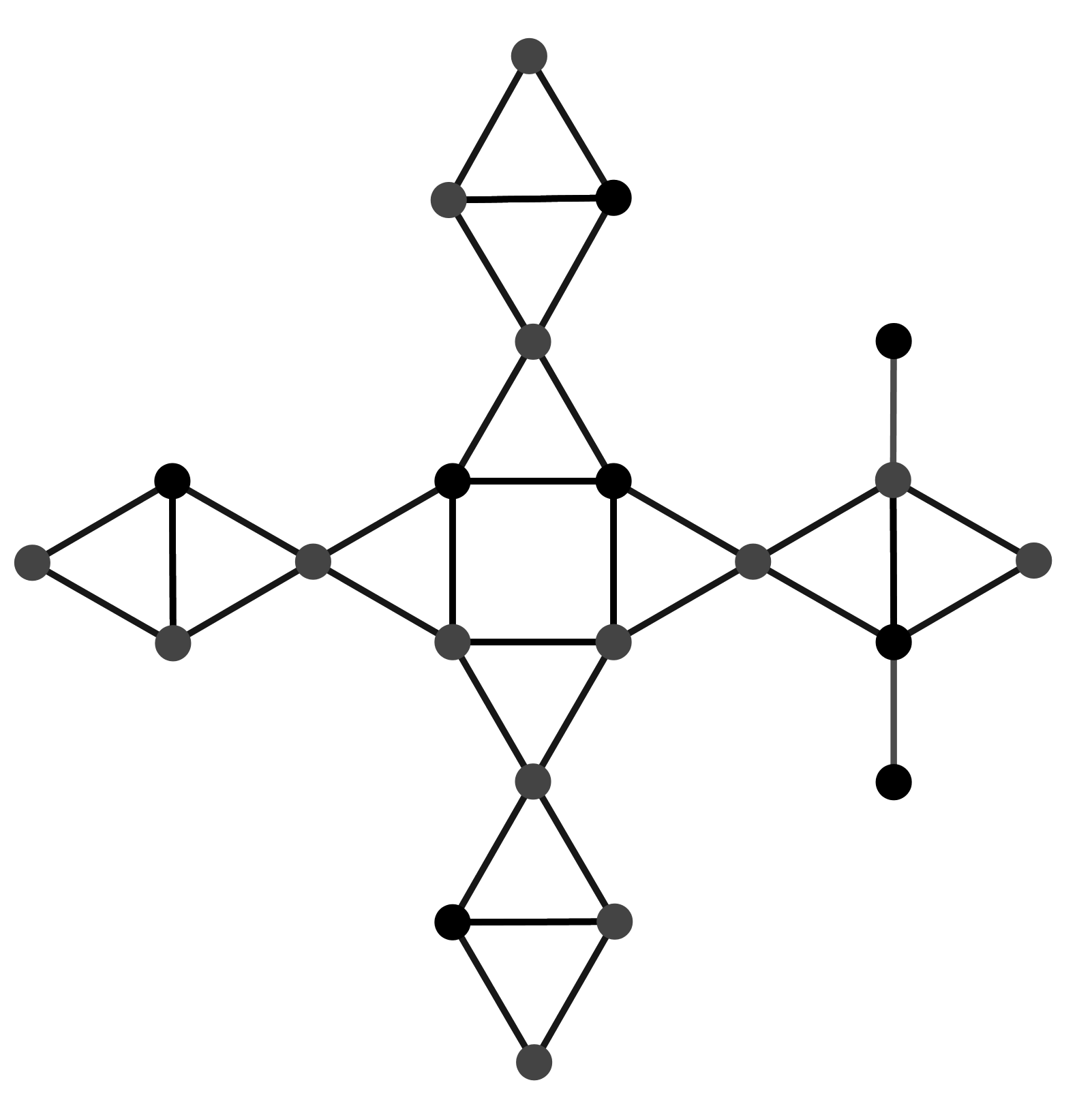}
\caption{Our ``bee" graph.}
\label{fig: g3}
\end{figure}
Note, one or both of the antennae may attach to a wing to form a pentagon face. 
Applying Lemma \ref{th:5-face}, we find that, $\mathcal C_5(v) \geq 4$ for any vertex $v$ on the square face, while $\mathcal C_5(v') \leq 3$ for the vertex $v'$ at the ``neck" of the bee. This is contrary to Lemma \ref{th:cycles}.

Combining this and Lemma \ref{th:4-face}, we have shown that the only $4$-cycles we find are the boundary of the union of two triangular faces. Since there are no square faces, there must be at least one pentagon face. The graph self-destructs upon any attempt to draw the $1$-neighborhood of a pentagon face.

\subsubsection*{Sub-case: $\mathcal C_4 = 2$} There are a few situations we must rule out. We first establish that each vertex must be incident to a square face. Suppose $v$ is not adjacent to a square face. Then, $r(v) = (3, 3, m, m')$ or $(3, m, 3, m')$, where $m, m' \geq 5$. The first is incident to only one $4$-cycle, so is ruled out. We try to draw the graph starting at this vertex and obtain the one in Figure \ref{fig: g4}, which forces a vertex to lie on at least three 4-cycles.
\begin{figure}
\includegraphics[width=0.5\textwidth]{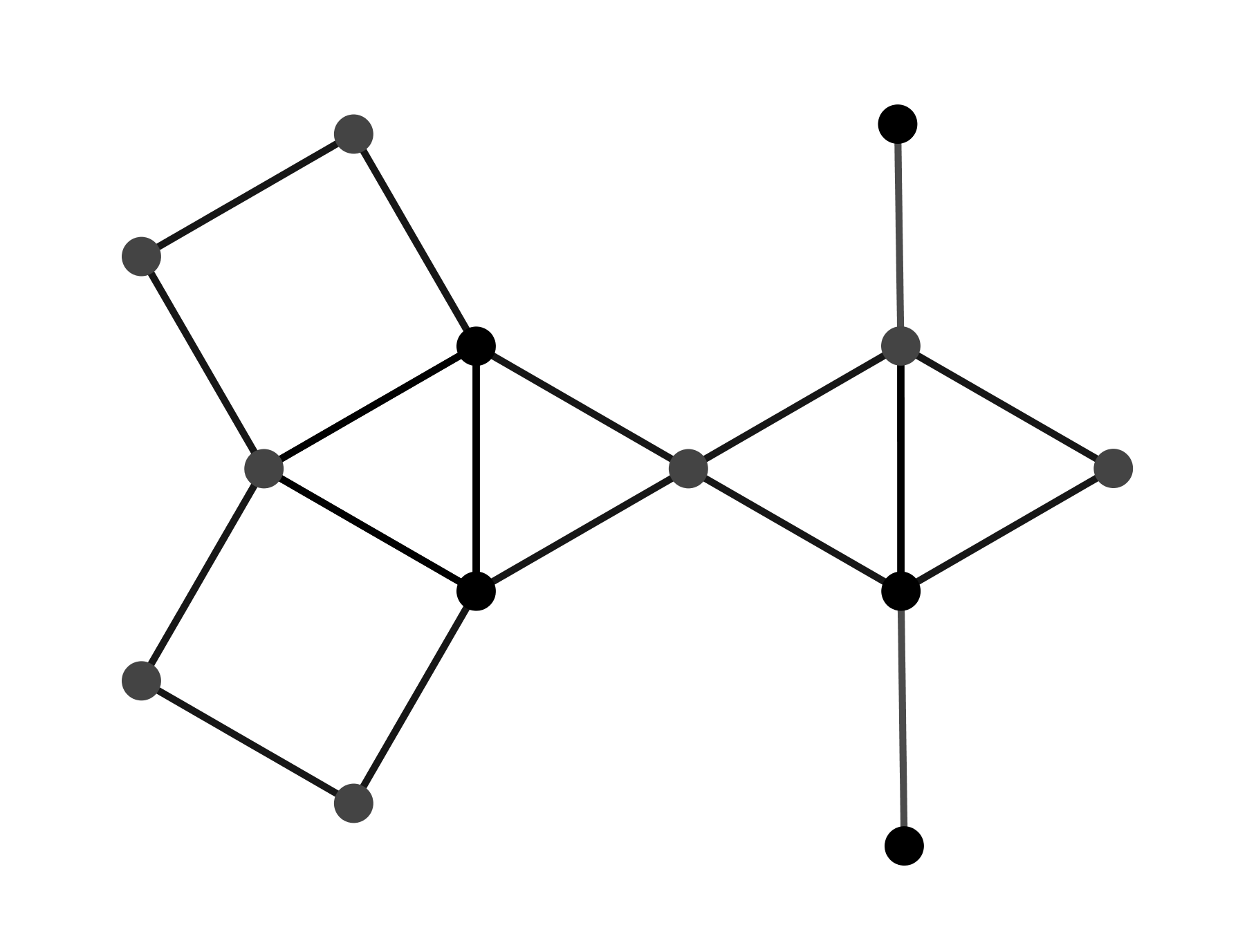}
\caption{The left-center vertex is forced to be incident to three $4$-cycles.}
\label{fig: g4}
\end{figure}

With this in hand, any construction of a graph with a vertex with the label $(3, 3, 4, m)$ for $m \geq 5$ fails after drawing a few faces. Hence, every vertex of our graph must have label $(3, 4, 3, 4)$, and we obtain the $(3, 4, 3, 4)$-solid.

\subsubsection*{Sub-case: $\mathcal C_4 \geq 3$} The only way a vertex can be incident to four $4$-cycles is if its $2$-neighborhood is that in Figure \ref{fig: g5}. However, this vertex's neighbors are only ever incident to three $4$-cycles.
\begin{figure}
	\includegraphics[width=0.5\textwidth]{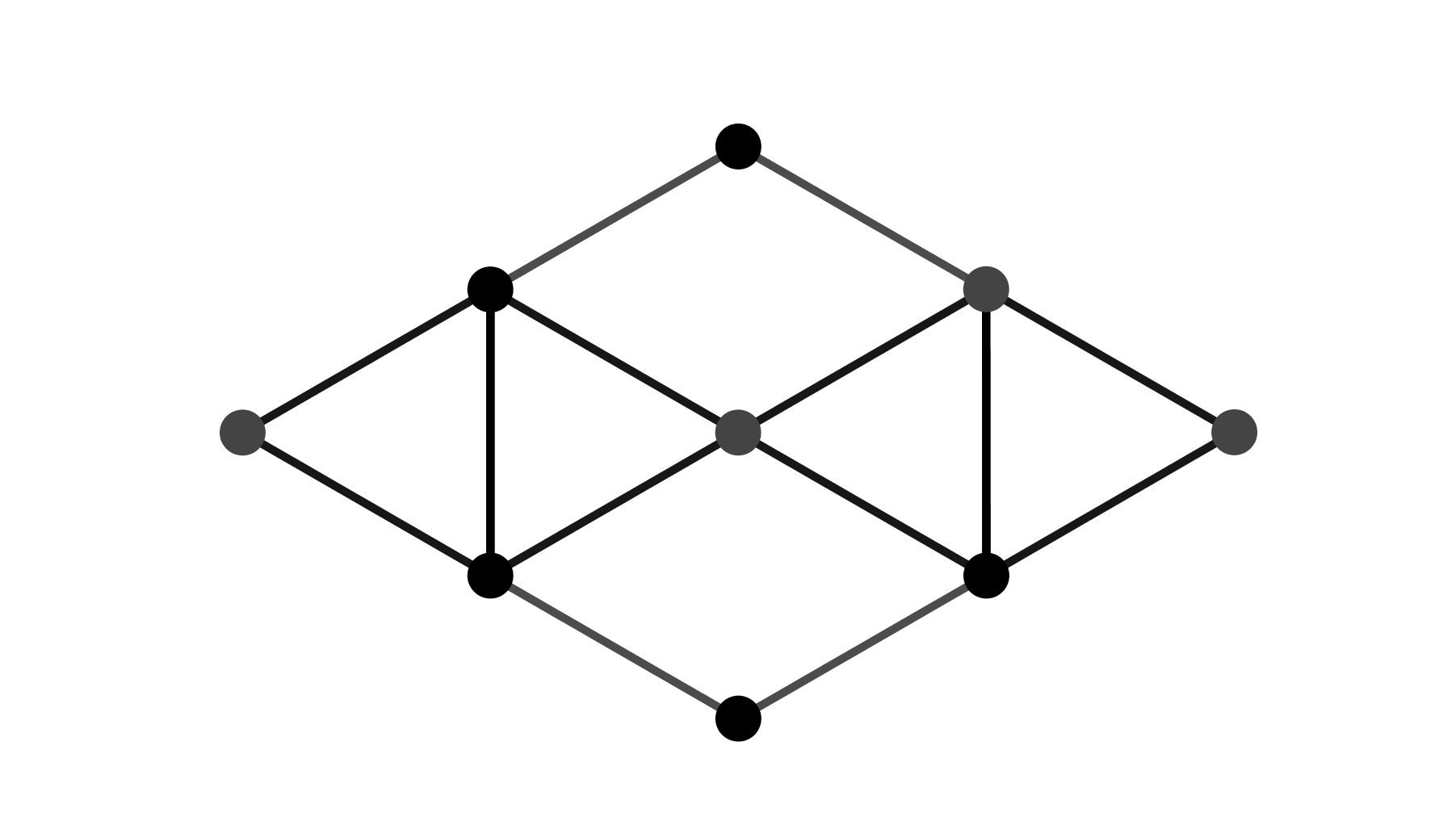}
	\caption{$2$-neighborhood of a vertex incident to four $4$-cycles.}
	\label{fig: g5}
\end{figure}
Hence, all that remains is to exclude the case $\mathcal C_4 = 3$.

A similar argument as the one above excludes the case where $v$ is incident to a face of length $\geq 5$. Hence, for each vertex $v$, we have either $r(v) = (3, 4, 3, 4)$ or $r(v) = (3, 3, 4, 4)$. There is a unique graph with these properties, depicted in Figure \ref{fig: g6}.
\begin{figure}
	\includegraphics[width=0.5\textwidth]{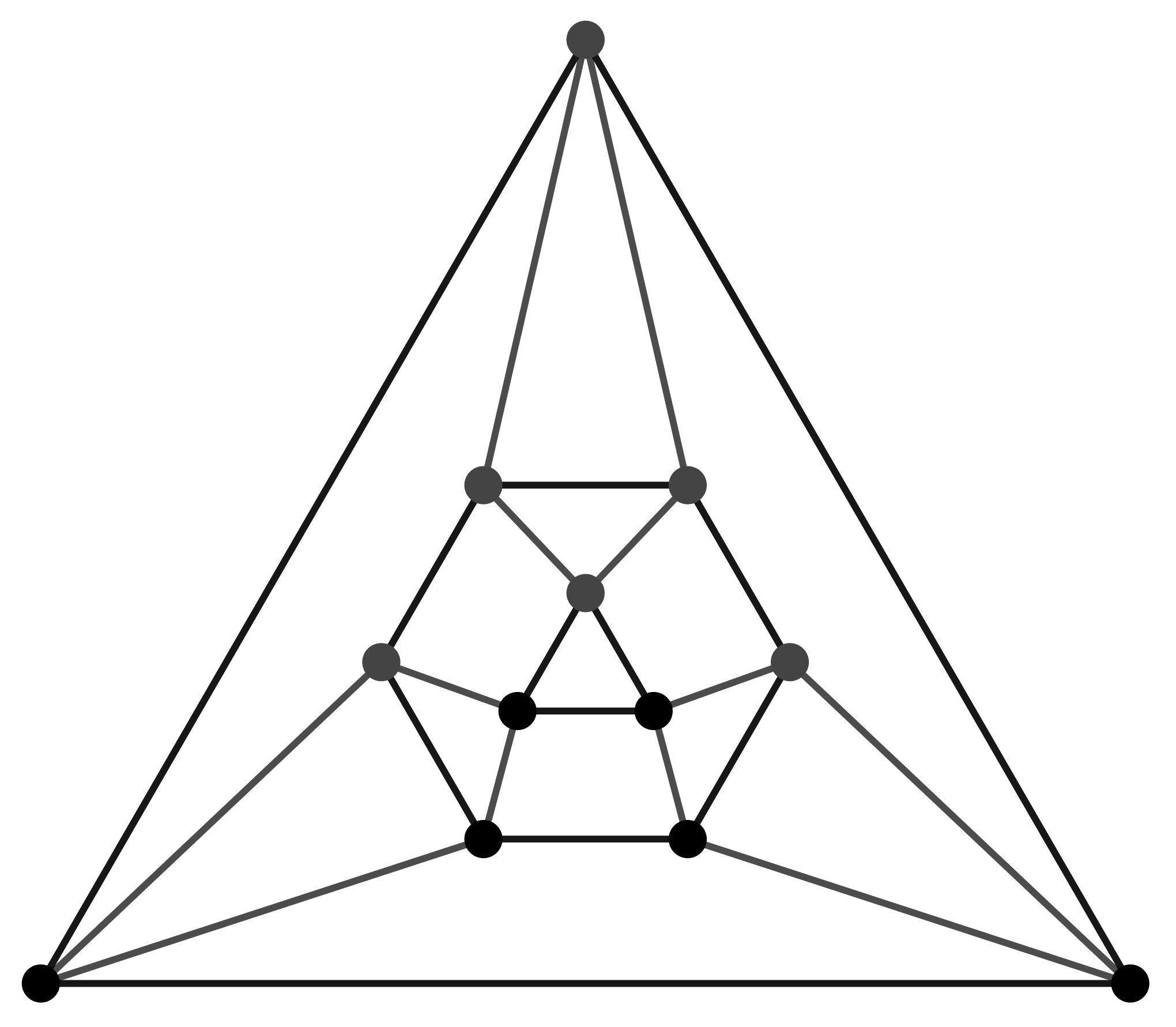}
	\caption{The graph with exactly three $4$-cycles through each vertex.}
	\label{fig: g6}
\end{figure}
This graph, however, is not walk-regular. The three innermost and outermost vertices all have seven incident $5$-cycles, while the other vertices have eight.

\subsection*{Case (13)}($d = 4, r_1 = 3, k = 1$)
As in the proof of the previous case, we subdivide the argument according to $\mathcal C_4$, the number of 4-cycles through each vertex. Since there is exactly 1 triangular face incident to each vertex, $\mathcal C_4$ is exactly the number of square faces around it. 
By \eqref{eq:3.1}, we must have $p(v_1)>1$, and thus $\mathcal C_4 \geq 2$.

\subsubsection*{Sub-case: $\mathcal C_4 = 3$} This case is straightforward since we already know that $r(v)\equiv(3,4,4,4)$. We remark that, however, when one wants to reconstruct the $(3,4,4,4)$-solid, there are two possible graphs, the other one being the twisted $(3,4,4,4)$-solid graph. Nonetheless, there are vertices on the twisted $(3,4,4,4)$-solid that have different numbers of closed walks of length 7, and thus, this twisted case should be ruled out. See Figure \ref{3444}.
\begin{figure}
    \includegraphics[width=0.85\textwidth]{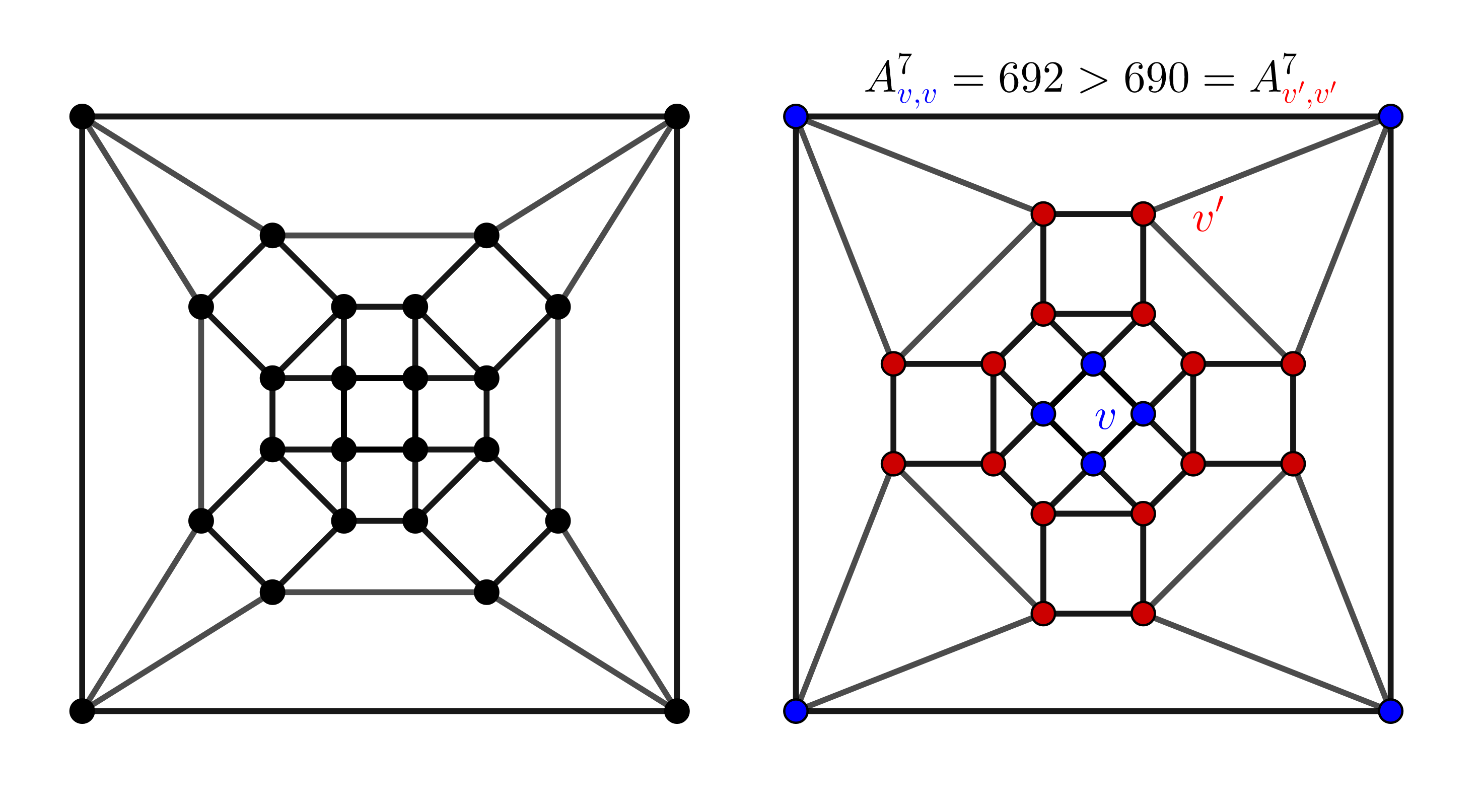}
    \caption{(3,4,4,4)-solid and twisted (3,4,4,4)-solid.}\label{3444}
\end{figure}

\subsubsection*{Sub-case: $\mathcal C_4 = 2$} In this sub-case, $r(v_1)$ is either $(3,4,4,5)$ or $(3,4,5,4)$. We first consider the case $r(v_1)=(3,4,4,5)$. One can try to draw such a graph starting from $v_1$ and end up with a neighborhood of the shape as in Figure \ref{fig: 3454-1}.
\begin{figure}
    \includegraphics[width=0.55\textwidth]{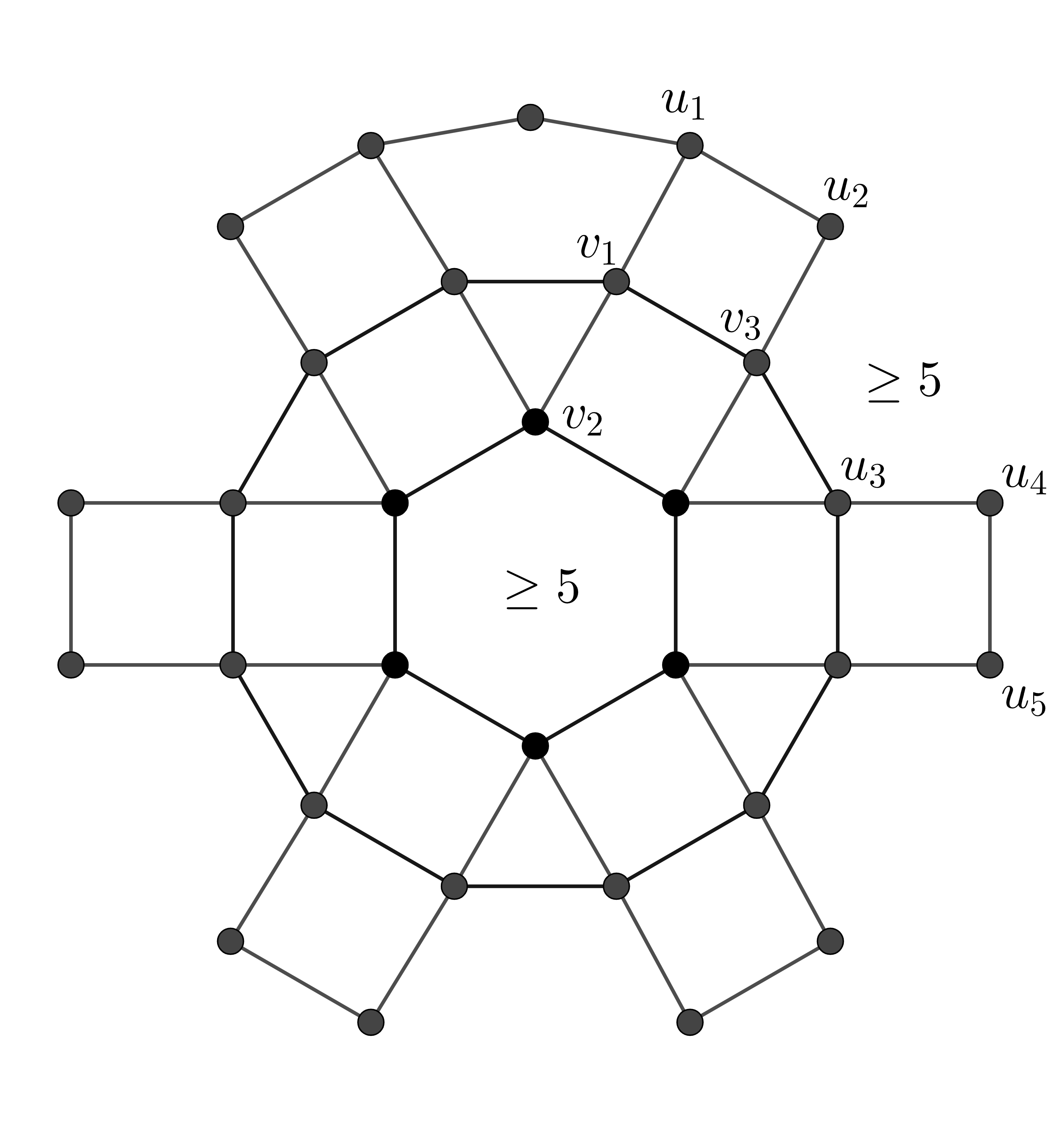}
    \caption{$r(v_1)=(3,4,4,5)$}\label{fig: 3454-1}
\end{figure}
 Here, vertex $v_1$ is what we started with. We claim that the unknown face incident to vertex $v_3$ must be of length 5. Then a similar argument will show that all unknown faces within the 2-neighborhood of the center face must be of length 5. To do so, we count the 5-cycles incident to $v_1$ and to $v_3$. 
Note that $v_1$ and $v_3$ always share the same number of ``houses'' (i.e., the 5-cycle with a chord), and $v_1$ is already incident to a pentagon face. To satisfy $\mathcal C_5(v_1) = \mathcal C_5(v_3)$, we must have a chordless 5-cycle $C$ through $v_3$, and there are three possible situations: (i) $C = \overline{u_1u_2v_3u_3u_4u_1}$, (ii) $C=\overline{u_2v_3u_3u_4u_5u_2}$, and (iii) $u_2,u_4$ have a common neighbor $w$ such that $C = \overline{wu_2v_3u_3u_4w}$. 

Among these situations, (i) is simply ruled out as $\{u_2,u_4\}$ would be a cut-set of the graph. For (ii), the 3-connectedness yields that each of $u_2$ and $u_5$ has an incident edge inside $C$, and these two edges must have a common endpoint $x$ to bound the incident triangle of $u_2$. However, there is still a cut-set $\{x,u_4\}$ of the graph. 
Now we are left with (iii). Using a similar argument, it is impossible to have $\geq 3$ inner neighbors of $C = \overline{wu_2v_3u_3u_4w}$, which is a contradiction. Hence, this $C$ must be a face with length $5$, and thus the claim is justified.

Next, we claim that we will get a contradiction if the center face is of length $\geq 6$. In this case, to have $\mathcal C_5(v_1) = \mathcal C_5(v_2)$, we cannot create more ``houses'' in the outer layer for vertex $v_1$. This forces us to get three square faces in a row. After drawing all the square faces we arrive at Figure \ref{fig: 3454-2}, which is impossible since there are vertices incident to more than 2 triangles.
\begin{figure}
    \includegraphics[width=0.55\textwidth]{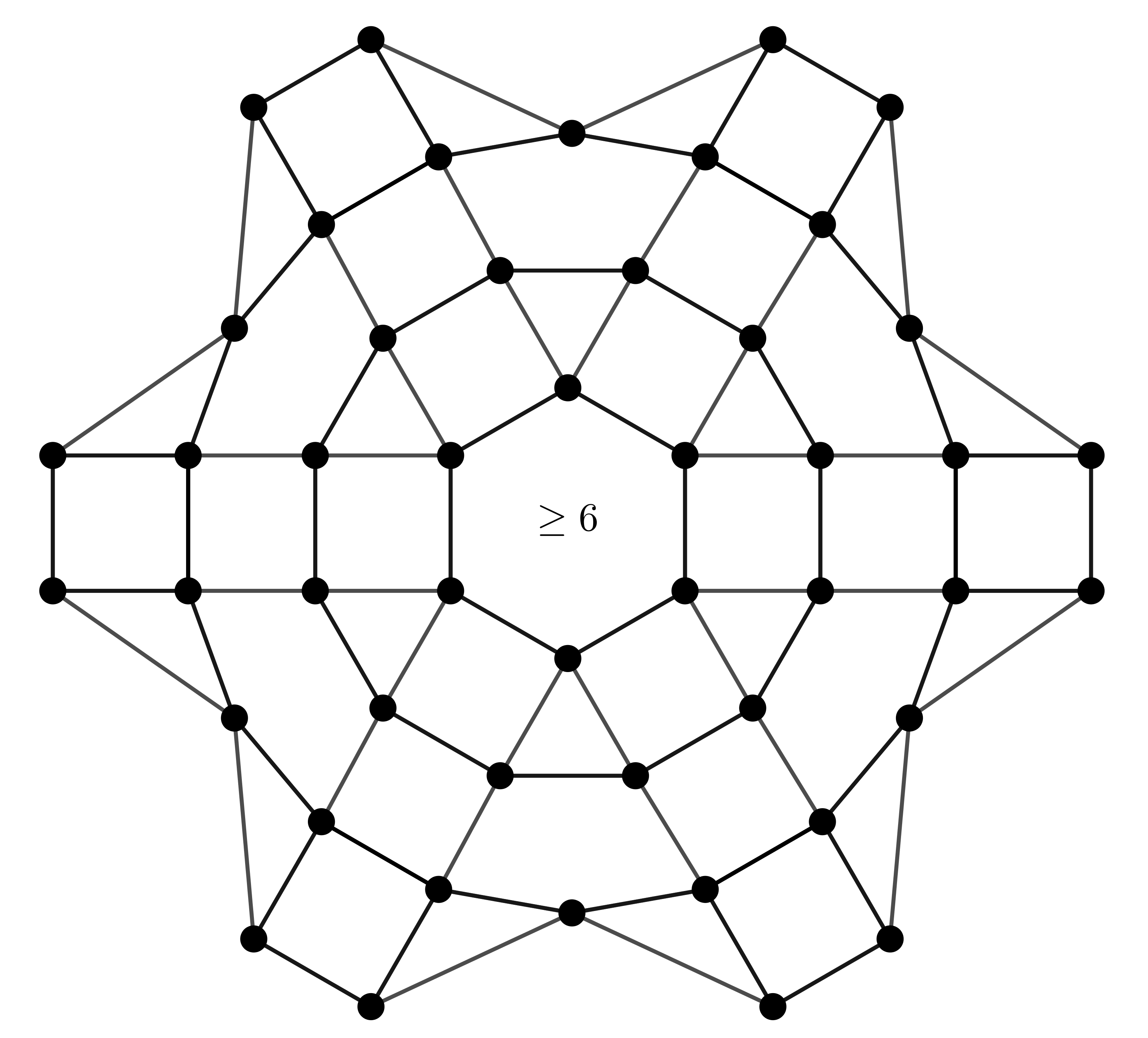}
    \caption{An impossible graph when the center face is of length $\geq 6$.}\label{fig: 3454-2}
\end{figure}
Therefore, we conclude that we must have a pentagon face at the center, and this in turn implies that we must add an additional ``house'' to vertex $v_1$, so we arrive at Figure \ref{fig: 3454-3}.

\begin{figure}
    \includegraphics[width=0.55\textwidth]{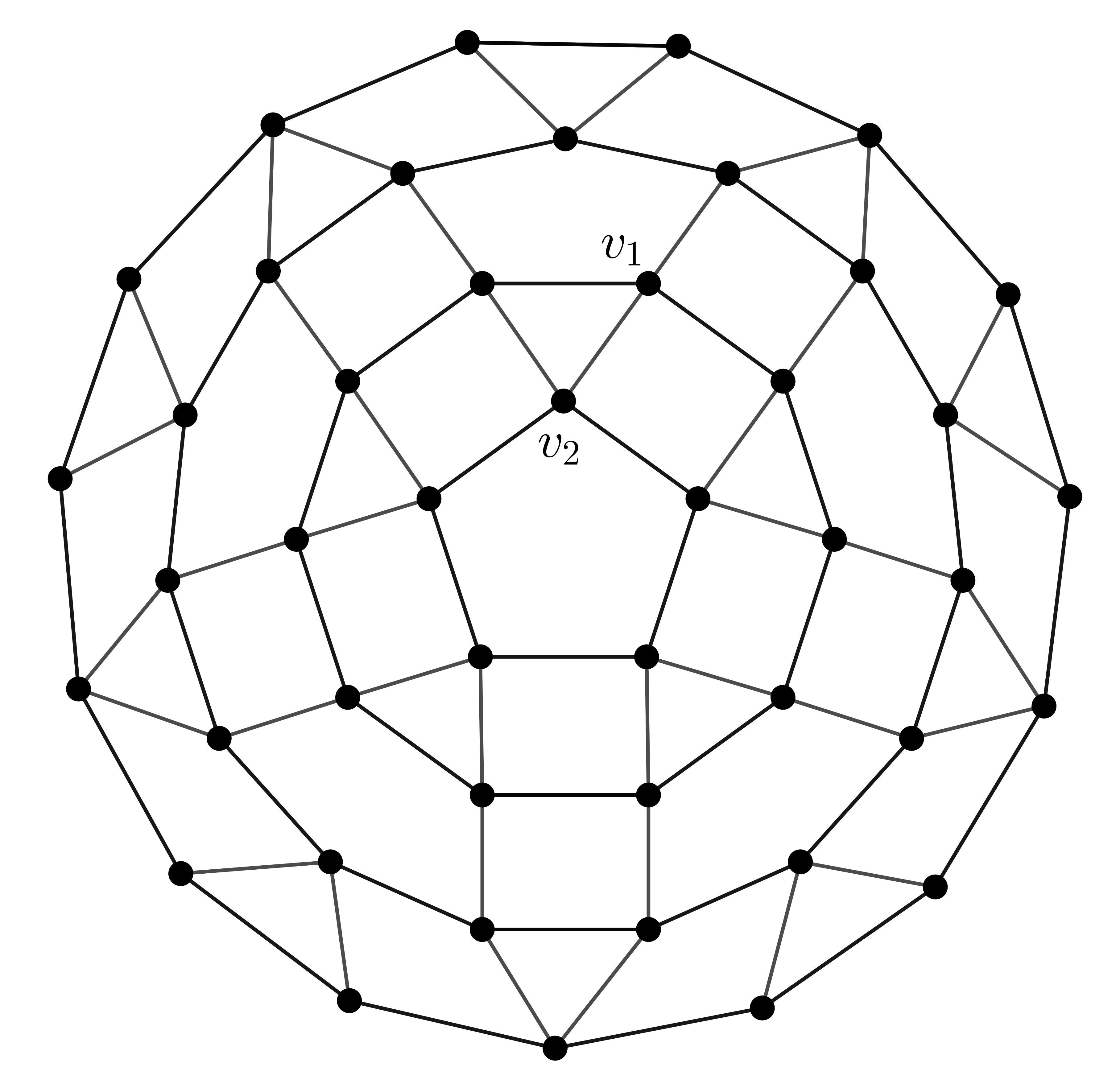}
    \caption{3-neighborhood of $v_1$ with $r(v_1)=(3,4,4,5).$}\label{fig: 3454-3}
\end{figure}
We remark that this is exactly part of the twisted $(3,4,5,4)$-solid. 
Counting the number of 6-cycles through $v_1$ and through $v_2$, we have $\mathcal C_6(v_1) =4\neq 5 = \mathcal C_6(v_2)$, which is contrary to Lemma \ref{th:cycles}. 
Therefore, there is no vertex $v$ satisfying $r(v)=(3,4,4,5)$ in this sub-case, and $v_1$ satisfies $r(v_1)=(3,4,5,4)$.

By a similar counting of the 5-cycles, we exclude $r(v) = (3,4,4,m)$ with $m > 5$, as such vertex $v$ gives $\mathcal C_5(v) \leq 4 < \mathcal C_5(v_1)$. Now the label of each $v$ is of the form $(3,4,m,4)$ with $m \geq 5$. Again, counting the 5-cycles in the 2-neighborhood of $v$, we have 
\[
\mathcal C_5(v) = \begin{cases}
    \, 6 \quad \text{if } r(v) = (3,4,5,4), \\
    \, 5 \quad \text{if } r(v) = (3,4,m,4) \text{ with } m > 5.
\end{cases}
\]
As a consequence, $r(v)\equiv (3,4,5,4)$, and we obtain the $(3,4,5,4)$-solid. 

We remark that the above discussion also completes the proof of Lemma \ref{th:5-face} for $d=4$.

\subsection{The 5-regular cases} Finally we assume that $G$ is 5-regular. The following discussion completes the proof of Lemma \ref{th:3-face}.

\begin{proof}[Proof of Lemma \ref{th:3-face} for $d=5$]
    Suppose that a triangle $C = \overline{a_1a_2a_3a_1}$ is not a face. The 3-connectedness implies that $d_*(a_i) \in \{1,2\}$ for all $i$. Again, we further assume that such $C$ is {\it minimal} among all the triangles that are not faces. Then each $v$ inside $C$ cannot be adjacent to all vertices $a_1,a_2,a_3$ of $C$, and around such $v$ there is at most one face with length $\geq 4$ since $\mathcal C_3(v) = \mathcal C_3(v_1) \geq 4$. Moreover, it follows that each edge inside $C$ lies on at least one triangle. In view of these facts, we can simply rule out the case $d_*(C) = (1,1,1)$, and eliminate the other three cases by drawing the adjacent vertices of $C$: 

    \medskip
    
    \noindent {\it Case: $d_*(C) = (1,1,2)$.} Note that the two adjacent vertices $b_1,b_2$ of $a_1, a_2$ inside $C$ cannot be identical, otherwise around this common neighbor there are at least two faces with length $\geq 4$. Since each of the edges $\overline{a_1b_1}, \overline{a_2b_2}$ belongs to some triangle, both $b_1$ and $b_2$ must be adjacent to $a_3$ (Figure \ref{fig: 3-cycle-1-1-2}). Now $\{b_1,b_2\}$ is a cut-set of the graph, and this violates the 3-connectedness.
    \begin{figure}
    \includegraphics[width=0.5\textwidth]{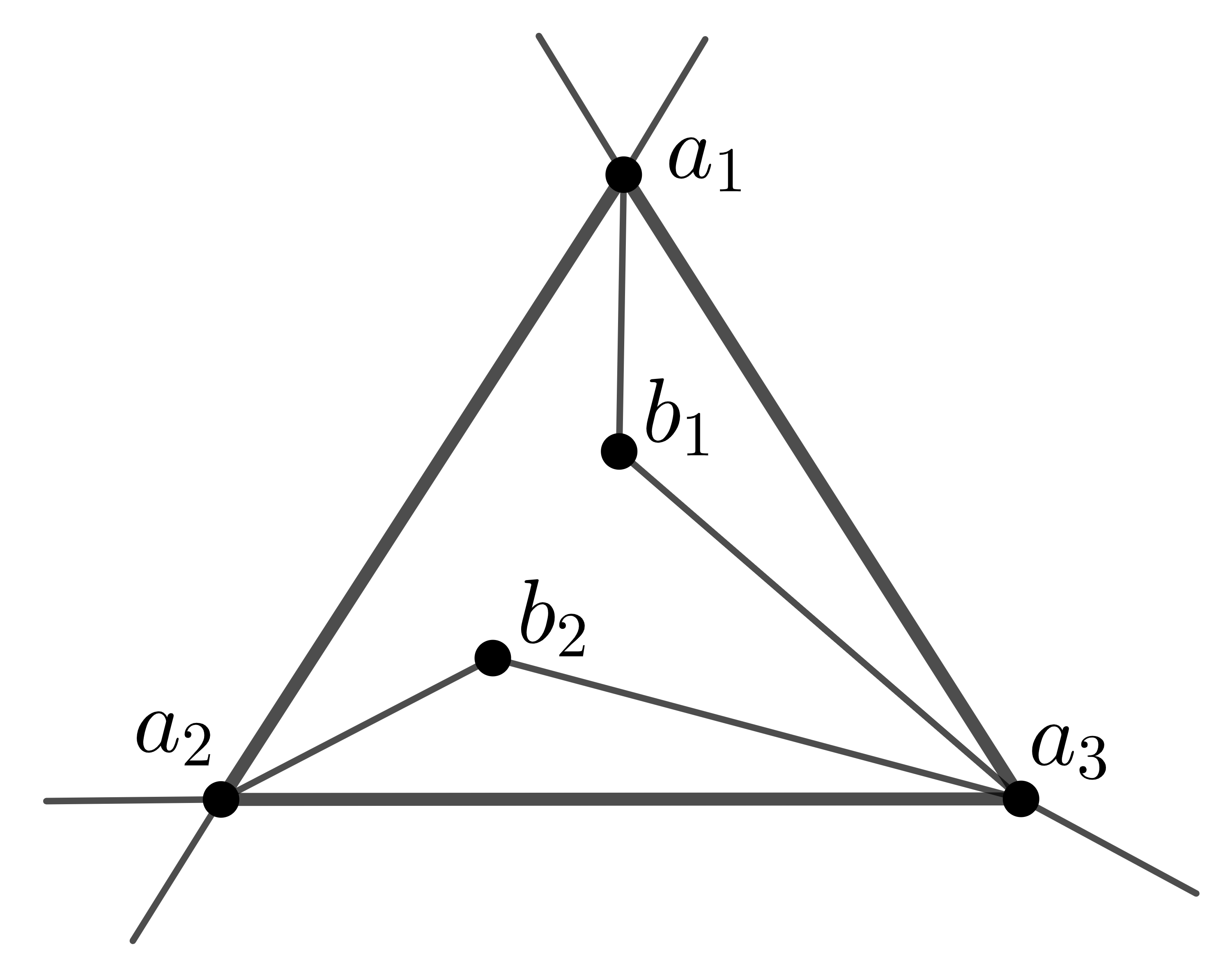}
    \caption{An impossible triangle with $d_*(C) = (1,1,2)$.}
    \label{fig: 3-cycle-1-1-2}
    \end{figure}

    \medskip
    
    \noindent {\it Case: $d_*(C) = (1,2,2)$.} Let $b_1$ be the neighbor of $a_1$ inside $C$. As $\mathcal C_3(a_1) \geq 4$ and $C$ is minimal, we can see that the edge $\overline{a_1b_1}$ belongs to exactly one triangle, and there are at least two incident triangles of $a_1$ outside $C$. Therefore, one of $a_2,a_3$ (say $a_2$) has a common neighbor with $a_1$ outside $C$. By the 3-connectedness there are at least three distinct neighbors of $\{a_1,a_2,a_3\}$ outside $C$, and thus $a_3$ has no common neighbors with $a_1$ and $a_2$ outside $C$. Now we have $\mathcal C_3 = 4$, which forces that $a_2, a_3$ have exactly two and three incident triangles inside $C$ respectively. We can draw four incident triangles $\overline{a_1 b_1 a_3 a_1}, \overline{a_2 b_3 b_2 a_2}, \overline{a_2 a_3 b_3 a_2}$ and $\overline{a_3 b_1 b_3 a_3} $ as in Figure \ref{fig: 3-cycle-1-2-2}. As $\mathcal C_3(b_1) =\mathcal C_3(b_2) = 4$, the undetermined neighbor of $b_3$ must be adjacent to $b_1$ and $b_2$. This leads to $\mathcal C_3(b_3) = 5$, a contradiction. 
    \begin{figure}
    \includegraphics[width=0.5\textwidth]{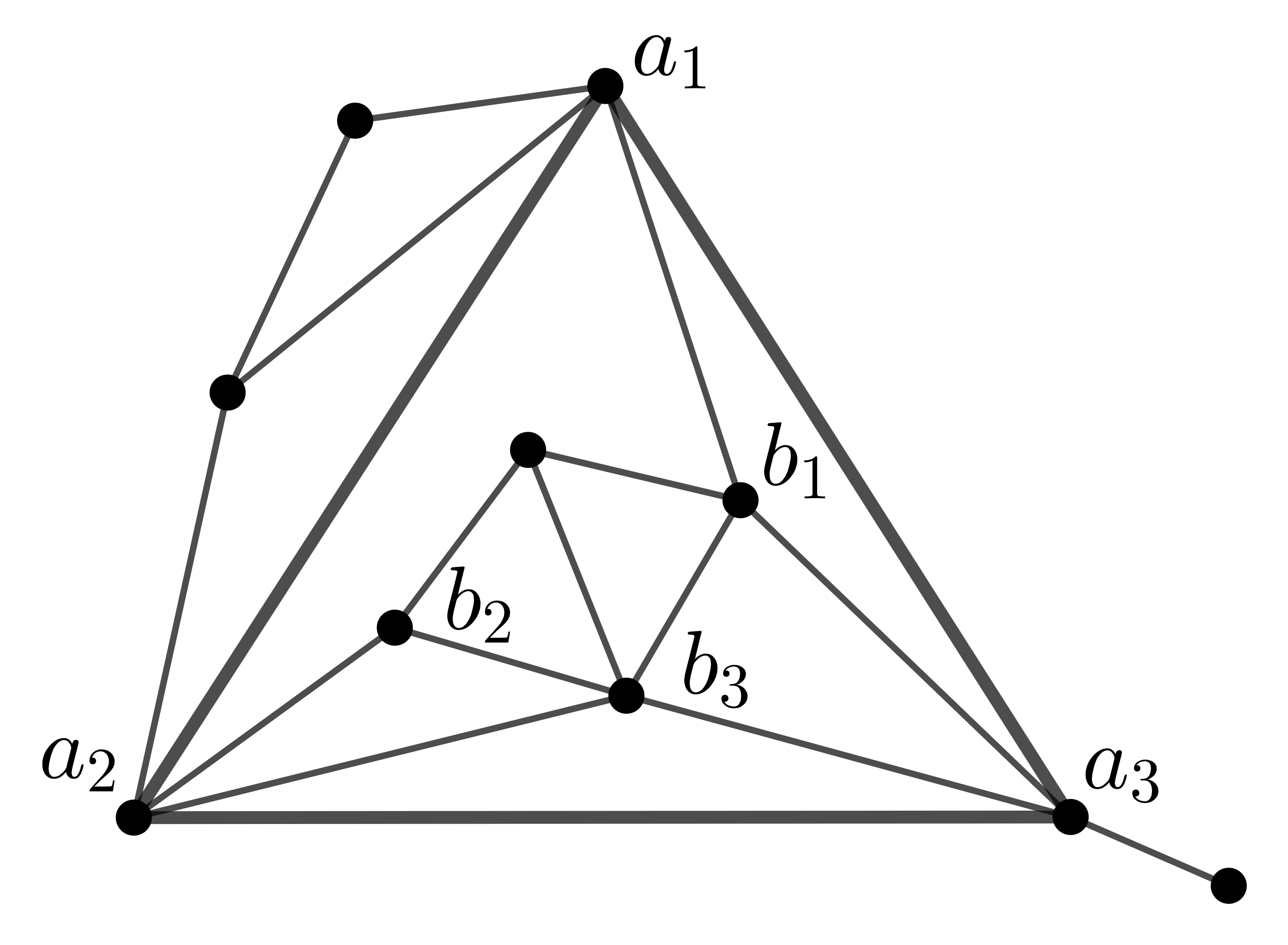}
    \caption{An impossible triangle with $d_*(C) = (1,2,2)$.}
    \label{fig: 3-cycle-1-2-2}
    \end{figure}

    \medskip
    
    \noindent {\it Case: $d_*(C) = (2,2,2)$.} Again by the 3-connectedness, to have three distinct neighbors of $\{a_1,a_2,a_3\}$ outside $C$,  no incident triangles outside $C$ are allowed. It follows from $\mathcal C_3 \geq 4$ that each $a_i$ has exactly three incident triangles inside $C$ as in Figure \ref{fig: 3-cycle-2-2-2}. However, the center triangle $\overline{b_1b_2b_3b_1}$ cannot be a face, and this violates the minimality of $C$. The proof of Lemma \ref{th:3-face} is complete.
    \begin{figure}
    \includegraphics[width=0.5\textwidth]{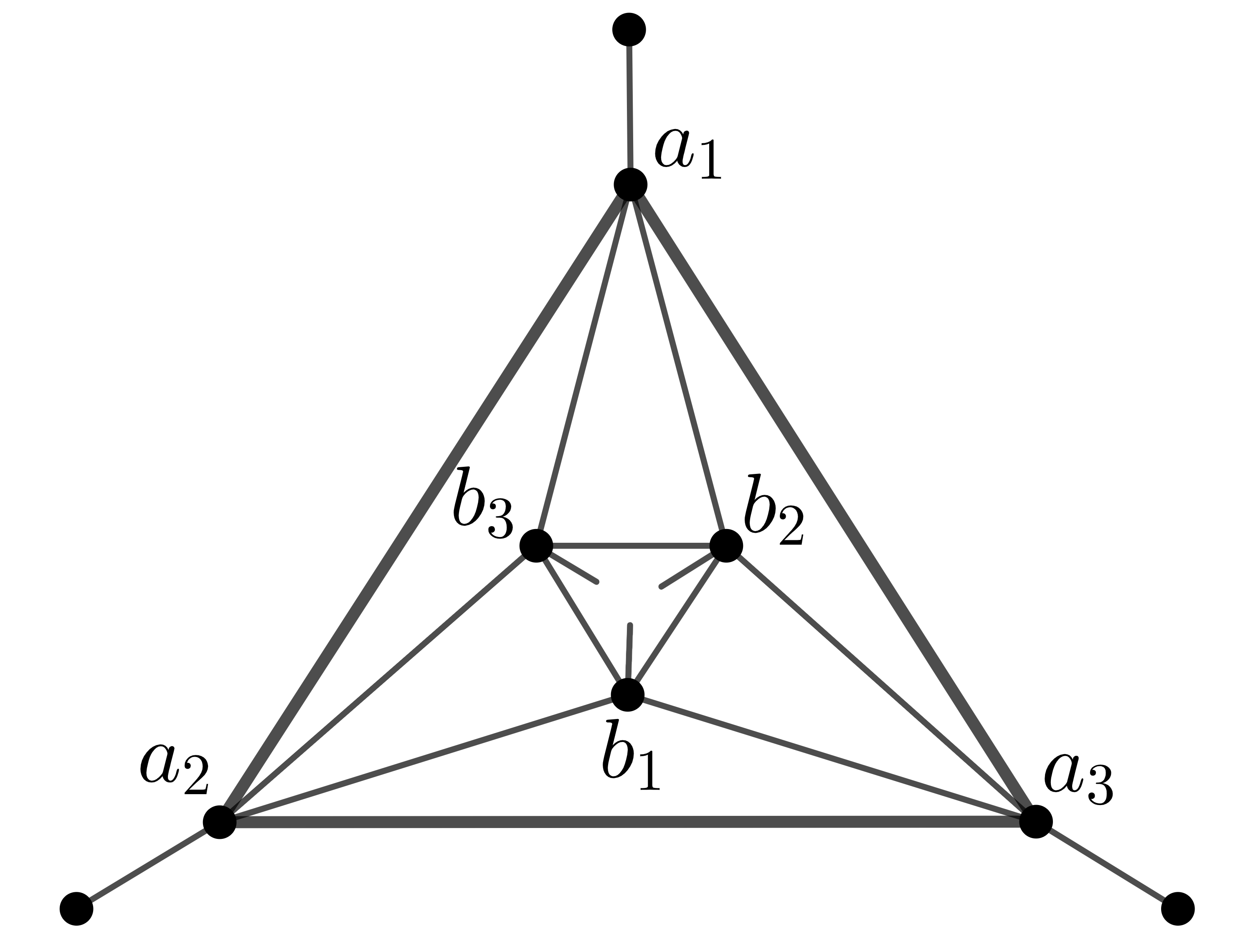}
    \caption{An impossible triangle with $d_*(C) = (2,2,2)$.}
    \label{fig: 3-cycle-2-2-2}
    \end{figure}
\end{proof}

As a result, in case (14) we have $r(v) \equiv (3,3,3,3,3)$, and the graph is isomorphic to icosahedron. To solve the last case ($d=5,r_1=3,k=4$), we also need the fact that every chordless 4-cycle must be a face there.  

\begin{proof}[Proof of Lemma \ref{th:4-face} for $d=5$]
    Suppose that a chordless $4$-cycle $C = \overline{a_1a_2a_3a_4a_1}$ is not a face, and such $C$ is minimal. This means that every chordless $4$-cycle inside $C$ must be a face, and thus, every vertex inside $C$ is adjacent to at most $2$ vertices on $C$. Note that $\mathcal C_3 = \mathcal C_3(v_1) = 4$ or $5$. If $\mathcal C_3 = 5$, then each $a_i$ has exactly two neighbors inside $C$, and $\{a_1,a_2,a_3,a_4\}$ must have a common neighbor outside $C$, which is a cut-vertex of the graph. Therefore, we are left with $\mathcal C_3 = 4$. For the case $d_*(a_1) = 0$, it is impossible to have $\geq 3$ distinct neighbors of $\{a_2,a_3,a_4\}$ inside $C$, which violates the 3-connectedness. Similarly, we rule out the case $d_*(a_i) = 3$. 
    
    Since the inner degree sequence is defined  lexicographically, if $d_*(a_1) = 2$, then we must have $d_*(C) = (2,2,2,2)$. Since there are at least three distinct neighbors outside $C$, the four incident edges outside $C$ provide at most one incident triangle, which is not enough to make $\mathcal C_3(a_i) = 4$ for each $i=1,2,3,4$. 
    
    Now the only possibility is $d_*(a_1) = 1$. Let $b_1$ be the neighbor of $a_1$ inside $C$. To have $\mathcal C_3(a_1) = 4$, $a_1$ has exactly three incident triangles outside $C$, and either $a_2$ or $a_4$ should be the second neighbor of $b_1$ on $C$. As the only face around $b_1$ with length $>3$ is bounded by $\overline{a_1b_1}$, the second neighbor of $b_1$ on $C$ should have two neighbors inside $C$. Since $d_*(a_2)\le d_*(a_4)$, we may assume that this neighbor of $b_1$ is $a_4$, and $d_*(a_4) = 2$. 
    
    If $d_*(a_2) = 1$, then following the same argument as above, the neighbor $b_2$ of $a_2$ inside $C$ is adjacent to $a_3$, and $d_*(a_3) = 2$. To have $\mathcal C_3(a_2) = \mathcal C_3(a_3) = \mathcal C_3(a_4) = 4$, there are three incident triangles of $a_2$ outside $C$, three incident triangles of $a_3$ as well as three of $a_4$ inside $C$, and thus $a_3,a_4$ have a common neighbor $b_3$ inside $C$ (Figure \ref{fig: 4-cycle-1-1-2-2}). Now either of the two undetermined faces around $b_3$ is a triangle. However, no matter which one is not a triangle, it becomes the second face around $b_1$ or $b_2$ with length $>3$, which is contrary to $\mathcal C_3(b_1) = \mathcal C_3(b_2) = 4$.
    \begin{figure}
    \includegraphics[width=0.5\textwidth]{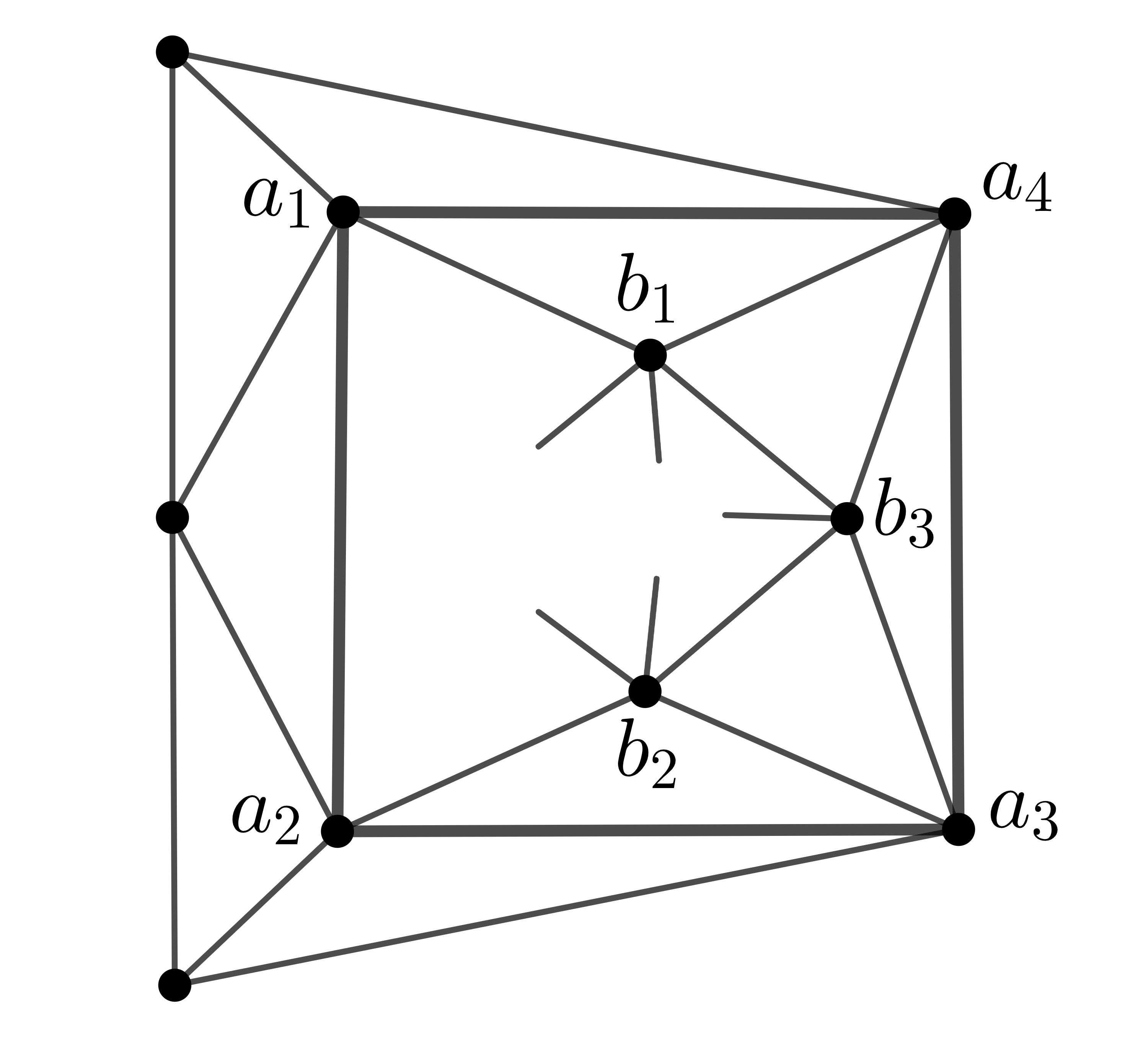}
    \caption{An impossible 4-cycle with $d_*(C) = (1,1,2,2)$.}
    \label{fig: 4-cycle-1-1-2-2}
    \end{figure}

    The last case we will rule out is $d_*(a_2) = 2$. As $\mathcal C_3(a_2) = 4$, $a_2$ has two incident triangles outside $C$, so that the common neighbor of $a_1$ and $a_2$ is adjacent to $a_3$. To have three distinct neighbors of $\{a_1,a_2,a_3,a_4\}$ outside $C$, $d_*(a_3)$ should be $1$, and hence no triangles are bounded by $\overline{a_3a_4}$, which violates $\mathcal C_3(a_3) = \mathcal C_3(a_4) = 4$. The proof of Lemma \ref{th:4-face} is complete.
\end{proof}

\subsection*{Case (15)} ($d=5,r_1=3,k=4$) As $p(v_1) > \frac 32$ by \eqref{eq:3.1}, $r_5(v_1) = 4$ or $5$. Since around each vertex four triangular faces are determined by Lemma \ref{th:3-face}, we are able to draw the 2-neighborhood of a face with length $\geq 4$, as in Figure \ref{fig: 3-3-3-3-4}. It is straightforward to count that 
\[
\mathcal C_4(v) = \begin{cases}
\, 7 \quad \text{if } r(v) = (3,3,3,3,4), \\
\, 6 \quad \text{if } r(v) = (3,3,3,3,m) \text{ with } m > 4.
\end{cases}
\]
In view of Lemma \ref{th:cycles}, $r(v) \equiv (3,3,3,3,4)$ and the graph is isomorphic to the $(3,3,3,3,4)$-solid provided that $r_5(v_1) = 4$.
\begin{figure}
\includegraphics[width=0.45\textwidth]{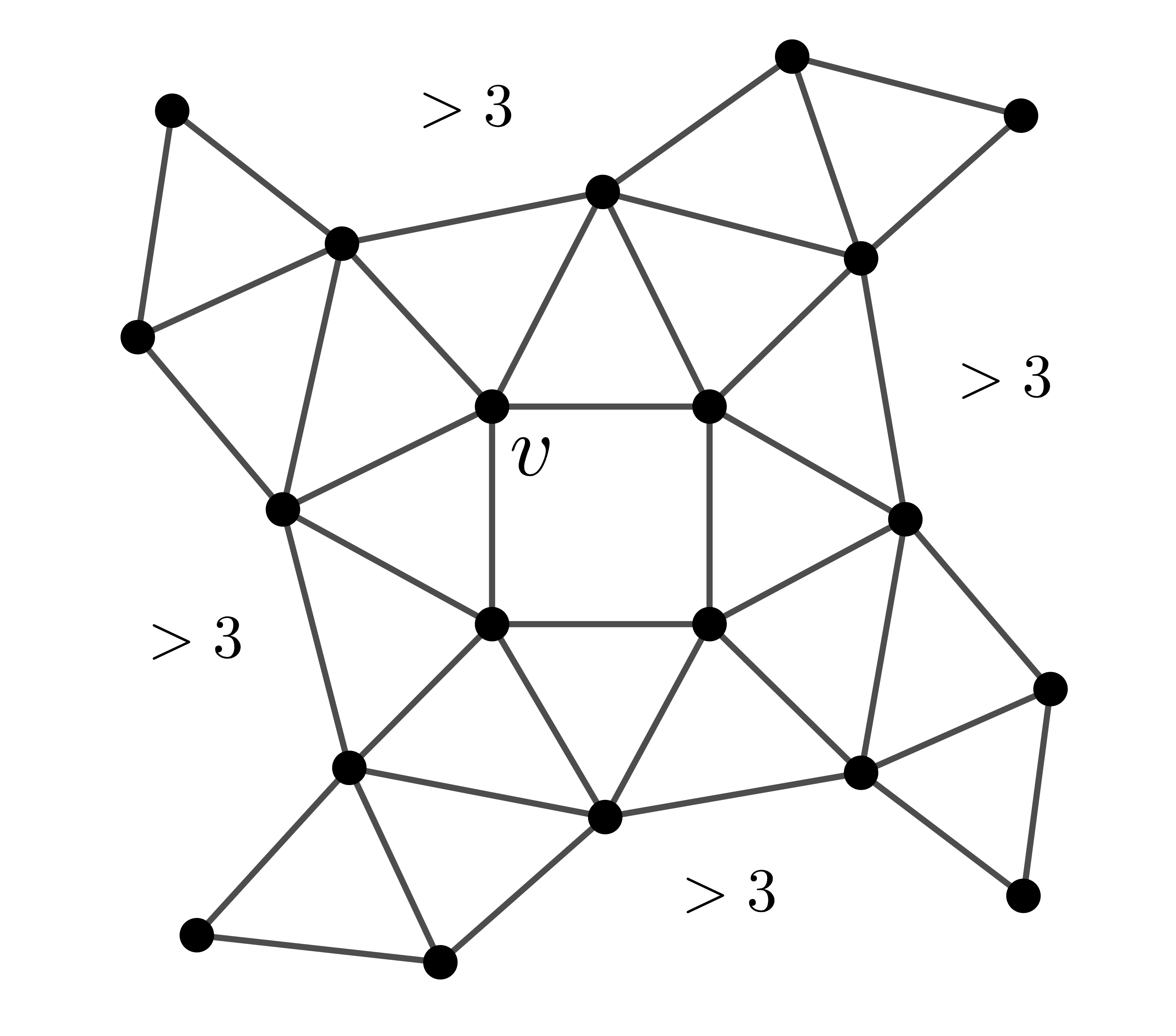} 
\includegraphics[width=0.45\textwidth]{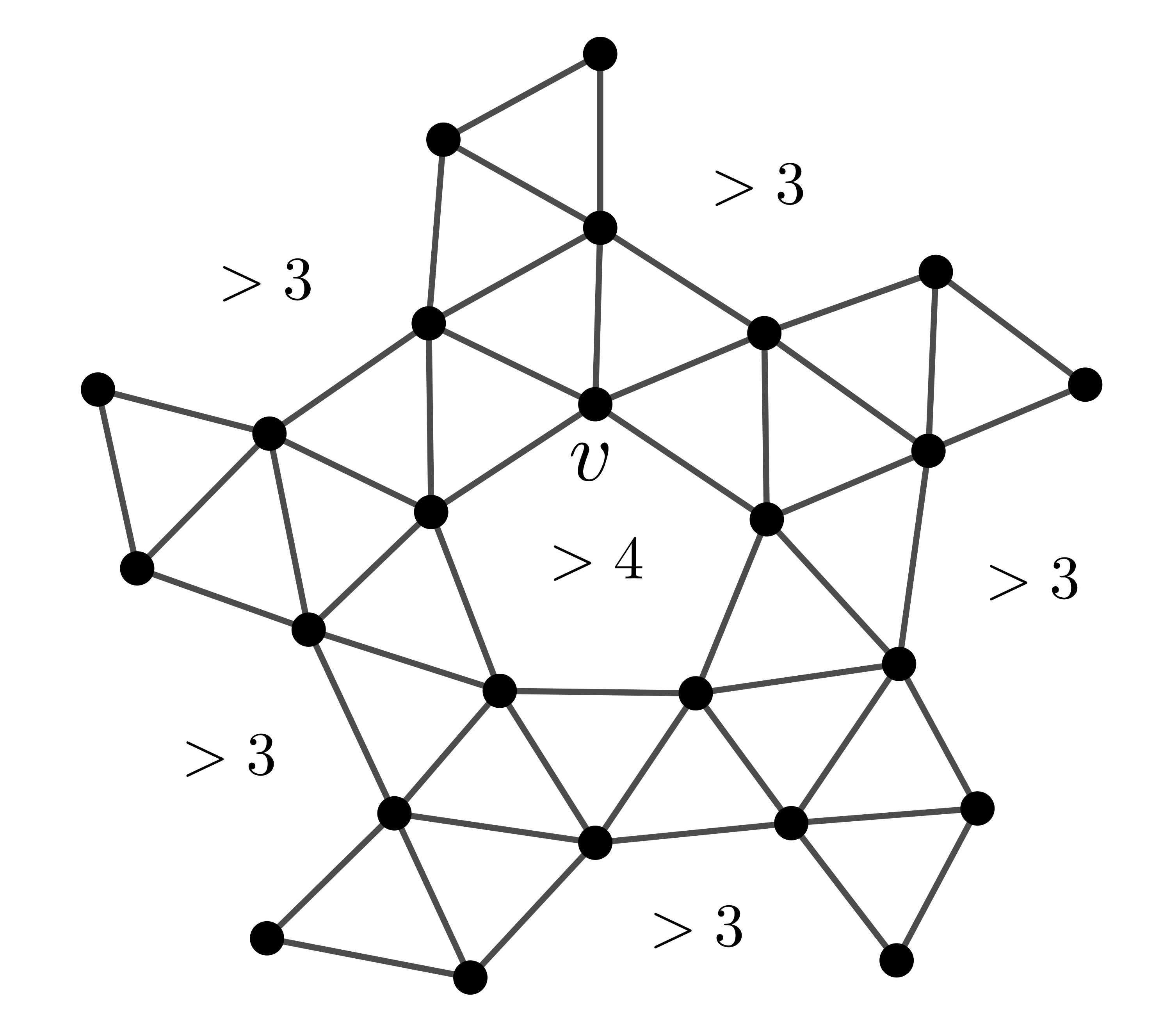} 
\caption{The $2$-neighborhood of faces with length $4$ and $>4$ respectively.}
\label{fig: 3-3-3-3-4}
\end{figure}

When $r_5(v_1) = 5$, each vertex $v$ lies on a face with length $\geq 5$. Note that in the 2-neighborhood of such face, the vertices in the 2-neighborhood of $v$ with undetermined neighbors are $u_1$, $u_2$ and $u_3$. As there are no faces with length $4$, $u_2$ is not adjacent to $u_3$. Moreover, $u_1$ is not adjacent to $u_3$ either, otherwise they form a 4-cycle that is not a face, which is contrary to Lemma \ref{th:4-face}. Therefore, each 5-cycle through $v$ can be found in the Figure \ref{fig: 3-3-3-3-5}, which is either a union of three adjacent triangles or a face with length $5$ (we remark that the proof of Lemma \ref{th:5-face} is complete here). A counting of these 5-cycles gives
\[
\mathcal C_5(v) = \begin{cases}
\, 11 \quad \text{if } r(v) = (3,3,3,3,5), \\
\, 10 \quad \text{if } r(v) = (3,3,3,3,m) \text{ with } m > 5.
\end{cases}
\]
Again by Lemma \ref{th:cycles}, we finally have $r(v) \equiv (3,3,3,3,5)$, and the graph is isomorphic to the $(3,3,3,3,5)$-solid. 
\begin{figure}
\includegraphics[width=0.45\textwidth]{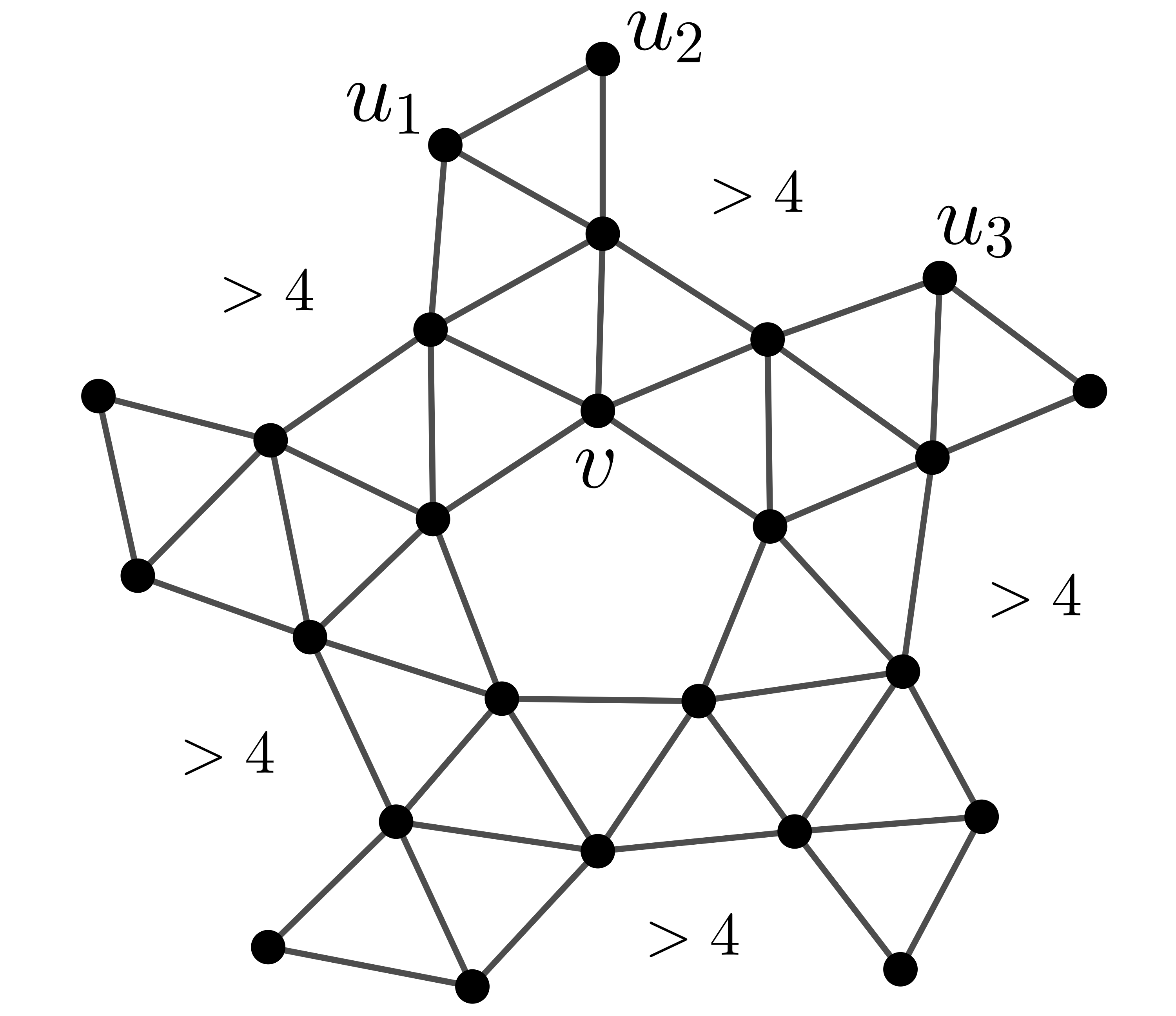} 
\includegraphics[width=0.45\textwidth]{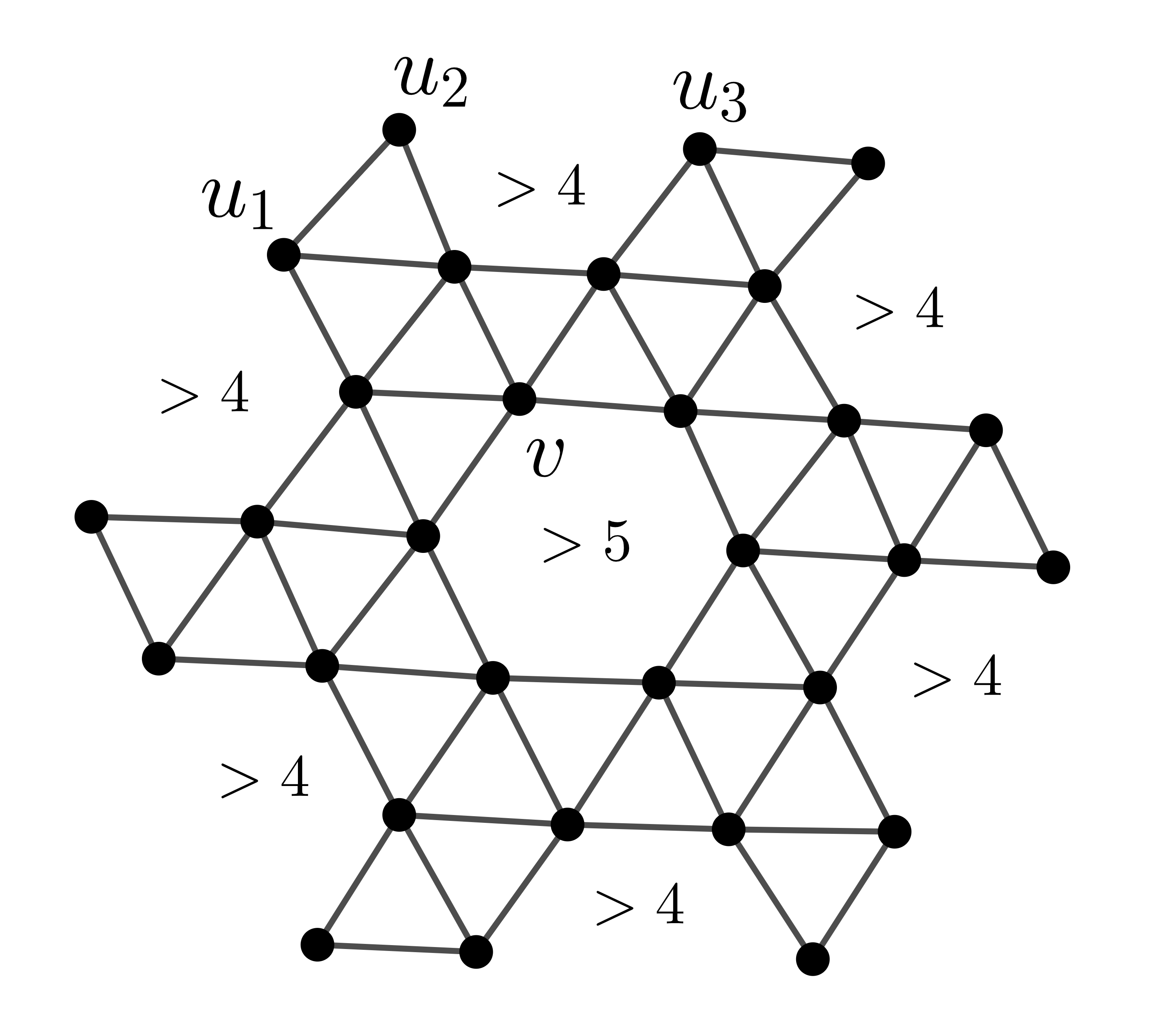} 
\caption{The $2$-neighborhood of faces with length $5$ and $>5$ respectively when $r_5(v_1) = 5$.}
\label{fig: 3-3-3-3-5}
\end{figure}

\bibliography{references}{} 
\bibliographystyle{alpha}

\end{document}